\documentclass[12pt]{article}
\usepackage[latin1]{inputenc}

\usepackage{fullpage}
\usepackage{amsfonts}
\usepackage{amssymb}
\usepackage{amsmath}
\usepackage{color}
\usepackage{stmaryrd}  
\usepackage{wasysym}   

\usepackage[colors]{optsys}

\numberwithin{theorem}{section}

\usepackage[numbers,sort&compress]{natbib}
\bibpunct[, ]{[}{]}{,}{n}{,}{,}
\makeatletter
\def\NAT@def@citea{\def\@citea{\NAT@separator}}
\makeatother

\title{Constant Along Primal Rays Conjugacies\\
  and the $l_0$ Pseudonorm}

\author{Jean-Philippe Chancelier 
  and Michel De Lara\footnote{michel.delara@enpc.fr}
  \\ CERMICS, Ecole des Ponts, Marne-la-Vall\'ee, France}

\begin{document}

\maketitle

\begin{abstract}
 The so-called \lzeropseudonorm\ on~$\RR^d$
 counts the number of nonzero components of a vector.
 For exact sparse optimization problems --- with the \lzeropseudonorm\
 standing either as criterion or in the constraints ---
 the Fenchel conjugacy fails to provide relevant analysis.
  In this paper, we display a class of conjugacies that 
  are suitable for the \lzeropseudonorm. 
  For this purpose, we suppose given a (source) norm on~$\RR^d$.
  With this norm, we define, 
  on the one hand, a sequence of so-called coordinate-$k$ norms
  and, on the other hand, a coupling between~$\RR^d$ and itself, 
  called Capra (constant along primal rays).
  Then, we provide formulas for the \Capra-conjugate and biconjugate,
  and for the \Capra\ subdifferentials, 
  of functions of the \lzeropseudonorm,
  in terms of the coordinate-$k$ norms.
  As an application, we provide a new family of lower bounds for 
  the \lzeropseudonorm, as a fraction between two norms,
  the denominator being any norm.
\end{abstract}





\section{Introduction}

The \emph{counting function}, also called \emph{cardinality function}
or \emph{\lzeropseudonorm}, 
counts the number of nonzero components of a vector in~$\RR^d$.
The \lzeropseudonorm\ measures the sparsity of a vector,
and the literature in sparse optimization that mentions it is plethoric. 
However, because of its combinatorial nature, 
the problems of minimizing the \lzeropseudonorm\ under constraints
or of minimizing a criterion under $k$-sparsity constraint
(\lzeropseudonorm\ less than a given integer~$k$) 
are usually not tackled as such. Most of the literature in sparse
optimization studies surrogate problems where the \lzeropseudonorm\
either enters a penalization term or is replaced by a regularizing term.
We refer the reader to \cite{Nikolova:2016} that provides a brief tour of the
literature dealing with least squares minimization constrained by $k$-sparsity,
and to \cite{Hiriart-Urruty-Le:2013} for a survey of the rank
function of a matrix, that shares many properties with the
\lzeropseudonorm.

Conjugacies, and more generally dualities, are a powerful tool to tackle classes of
optimization problems. The Fenchel conjugacy plays a central role in analyzing
solutions of convex problems (and beyond) \cite{Rockafellar:1974}.
However, it fails to provide relevant analysis for
optimization problems involving the \lzeropseudonorm.
Indeed, the Fenchel biconjugate of the characteristic function of the level sets
of the \lzeropseudonorm\ is zero, and
the Fenchel biconjugate of the \lzeropseudonorm\ is also zero.
The field of generalized convexity goes beyond the Fenchel conjugacy 
and convex functions and displays conjugacies adapted to analyze
classes of functions such as 
increasing positive homogeneous, difference of convex, quasi-convex,
increasing and convex-along-rays.
For more details on the theory, and more examples, 
we refer the reader to the books \cite{Singer:1997,Rubinov:2000}
and to the nice introduction paper \cite{Martinez-Legaz:2005}.

To our knowledge, none of the conjugacies in the literature
is adapted to the \lzeropseudonorm\ (the \lzeropseudonorm\ is convex-along-rays according to the definition
in~\cite{Rubinov-Glover:1999} but not in \cite{Singer:1997}, and
calculation shows that the \lzeropseudonorm\ is not convex for the conjugacy in~\cite{Rubinov-Glover:1999}).
In this paper, we study the \lzeropseudonorm\ as such and we 
display a suitable class of conjugacies. 
We extend results of~\cite{Chancelier-DeLara:2021_ECAPRA_JCA} beyond the special
Euclidian norm setting.

The paper is organized as follows.
In Sect.~\ref{The_lzeropseudonorm_and_its_level_sets},
we recall the definition of the \lzeropseudonorm,
and we introduce the notion of sequence of norms on~$\RR^d$ 
that are (strictly or not) decreasingly graded with respect to the \lzeropseudonorm.
In Sect.~\ref{Coordinate-k_and_dual_coordinate-k_norms}, 
we introduce a sequence of coordinate-$k$ norms,
all generated from any (source) norm on~$\RR^d$,
and their dual norms.
In Sect.~\ref{The_Capra_conjugacy_and_the_lzeropseudonorm},
we define a so-called \Capra\ coupling between~$\RR^d$ and itself, 
that depends on any (source) norm on~$\RR^d$.
Then, we provide formulas for the \Capra-conjugate and biconjugate,
and for the \Capra\ subdifferentials, 
of functions of the \lzeropseudonorm\ 
(hence, in particular, of the \lzeropseudonorm\ itself
and of the characteristic functions of its level sets),
in terms of the coordinate-$k$ norms.
In Sect.~\ref{Norm_ratio_lower_bounds_for_the_l0_pseudonorm},
as an application, we provide a new family of lower bounds for 
the \lzeropseudonorm, as a fraction between two norms,
the denominator being any norm.
The Appendix~\ref{Appendix} gathers background on Fenchel-Moreau conjugacies.

\section{The \lzeropseudonorm\ and its level sets}
\label{The_lzeropseudonorm_and_its_level_sets}

First, we introduce basic notations regarding the \lzeropseudonorm.
Second, we recall the definition of a sequence of norms on~$\RR^d$
which is (strictly or not) decreasingly graded 
with respect to the \lzeropseudonorm\
(as introduced in the companion paper~\cite{Chancelier-DeLara:2020_OSM}).
We use the notation \( \ic{r,s}=\na{r,r+1,\ldots,s-1,s} \) for two integers
$r \leq s$.

\subsubsubsection{The \lzeropseudonorm}

For any vector \( \primal \in \RR^d \), 
\(  \Support{\primal} = \bset{ j \in \ic{1,d} }{\primal_j \not= 0 } 
\subset \ic{1,d} \) is the support of~\( \primal \).
The so-called \emph{\lzeropseudonorm} is the function
\( \lzero : \RR^d \to \ic{0,d} \)
defined by 
\begin{equation}
  \lzero\np{\primal} = \cardinal{ \Support{\primal} }
  = \textrm{number of nonzero components of } \primal
  \eqsepv \forall \primal \in \RR^d
  \eqfinv
  \label{eq:pseudo_norm_l0}  
\end{equation}
where $\cardinal{K}$ denotes the cardinal of 
a subset \( K \subset \ic{1,d} \). 
The \lzeropseudonorm\ shares three out of the four axioms of a norm:
nonnegativity, positivity except for \( \primal =0 \), subadditivity.
The axiom of 1-homogeneity does not hold true;
in contrast to norms, the \lzeropseudonorm\ is 0-homogeneous:
\begin{equation}
  \lzero\np{\rho\primal} = \lzero\np{\primal} \eqsepv \forall \rho \in \RR\backslash\{0\}
  \eqsepv \forall \primal \in \RR^d
  \eqfinp
  \label{eq:lzeropseudonorm_is_0-homogeneous}
\end{equation}

\subsubsubsection{The level sets of the \lzeropseudonorm}

The \lzeropseudonorm\ is used in exact sparse optimization problems of the form
\( \inf_{ \lzero\np{\primal} \leq k } \fonctionprimal\np{\primal} \).
Thus, we introduce 
\begin{subequations}
  \begin{align}
      \text{the \emph{level sets}}\qquad 
   \LevelSet{\lzero}{k} 
    &= 
      \defset{ \primal \in \RR^d }{ \lzero\np{\primal} \leq k }
      \eqsepv \forall k \in \ba{0,1,\ldots,d} 
      \eqfinv
      \label{eq:pseudonormlzero_level_set}
\\
      \text{and the \emph{level curves}}\qquad 
       \LevelCurve{\lzero}{k} 
    &= 
      \defset{ \primal \in \RR^d }{ \lzero\np{\primal} = k }
      \eqsepv \forall k \in \ba{0,1,\ldots,d} 
      \eqfinp
      \label{eq:pseudonormlzero_level_curve}
  \end{align}
\end{subequations}
For any subset \( K \subset \ic{1,d} \), 
we denote the subspace of~$\RR^d$ made of vectors
whose components vanish outside of~$K$ by\footnote{%
  Here, following notation from Game Theory, 
  we have denoted by $-K$ the complementary subset 
  of~$K$ in \( \ic{1,d} \): \( K \cup (-K) = \ic{1,d} \)
  and \( K \cap (-K) = \emptyset \).}
\begin{equation}
  \FlatRR_{K} = \RR^K \times \{0\}^{-K} =
  \bset{ \primal \in \RR^d }{ \primal_j=0 \eqsepv \forall j \not\in K } 
  \subset \RR^d 
  \eqfinv
  \label{eq:FlatRR}
\end{equation}
where \( \FlatRR_{\emptyset}=\{0\} \).
We denote by \( \pi_K : \RR^d \to \FlatRR_{K} \) 
the \emph{orthogonal projection mapping}
and, for any vector \( \primal \in \RR^d \), by 
\( \primal_K = \pi_K\np{\primal} \in \FlatRR_{K} \) 
the vector which coincides with~\( \primal \),
except for the components outside of~$K$ that are zero.
It is easily seen that the orthogonal projection mapping~$\pi_K$
is self-dual, giving
\begin{equation}
  \proscal{\primal_K}{\dual_K} 
  = \proscal{\primal_K}{\dual} 
  = \bscal{\pi_K\np{\primal}}{\dual} 
  = \bscal{\primal}{\pi_K\np{\dual}}
  = \proscal{\primal}{\dual_K} 
  \eqsepv 
  \forall \primal \in \RR^d 
  \eqsepv 
  \forall \dual \in \RR^d 
  \eqfinp
  \label{eq:orthogonal_projection_self-dual}
\end{equation}
The level sets of the \lzeropseudonorm\
in~\eqref{eq:pseudonormlzero_level_set}
are easily related to the subspaces~\( \FlatRR_{K} \) of~\( \RR^d \)
by\footnote{The notation \( \bigcup_{\cardinal{K} \leq k} \)
is a shorthand for 
\( \bigcup_{ { K \subset \ic{1,d}, \cardinal{K} \leq k}} \).}
\begin{equation}
  \LevelSet{\lzero}{k} 
  = 
  \bset{ \primal \in \RR^d }{ \lzero\np{\primal} \leq k }
  = \bigcup_{\cardinal{K} \leq k} \FlatRR_{K} 
  \eqsepv \forall k\in\ic{0,d} 
  \eqfinp
  \label{eq:level_set_pseudonormlzero}
\end{equation}

\subsubsubsection{Decreasingly graded sequence of norms with respect to the \lzeropseudonorm}

Now, we introduce the notion of sequences of norms 
that are, strictly or not,
decreasingly graded with respect to the \lzeropseudonorm:
in a sense, the monotone sequence detects 
the number of nonzero components of a vector in~$\RR^d$
when it becomes stationary.
In the following definition, 
\( \sequence{\TripleNorm{\cdot}_{k}}{k\in\ic{1,d}} \) denotes any 
sequence of norms on~$\RR^d$.

\begin{definition}(\cite[Definition~\ref{OSM-de:decreasingly_graded}]{Chancelier-DeLara:2020_OSM})
  We say that a sequence 
  \( \sequence{\TripleNorm{\cdot}_{k}}{k\in\ic{1,d}} \) of norms on~$\RR^d$
  is \emph{decreasingly graded} 
(resp. \emph{strictly decreasingly graded})
\wrt\ (with respect to) the \lzeropseudonorm\ if,
  for any \( \primal\in\RR^d \),
  one of the three following equivalent statements holds true.
  \begin{enumerate}
  \item 
    We have the implication (resp. equivalence), for any \( l\in\ic{1,d} \), 
    \begin{subequations}
      \begin{align}
        \lzero\np{\primal} = l 
        &\implies
        \TripleNorm{\primal}_{1} \geq \cdots \geq 
        \TripleNorm{\primal}_{l-1} \geq 
        \TripleNorm{{\primal}}_{l} = 
        \cdots =
        \TripleNorm{{\primal}}_{d} 
        \eqfinv
        \label{eq:decreasingly_graded_a}
    \\
 \left( \right. \text{resp.} \qquad      
        \lzero\np{\primal} = l 
        &\iff 
        \TripleNorm{\primal}_{1} \geq \cdots \geq 
        \TripleNorm{\primal}_{l-1} >
        \TripleNorm{{\primal}}_{l} = 
        \cdots =
        \TripleNorm{{\primal}}_{d}
        \eqfinp 
          \left. \right)
        \label{eq:strictly_decreasingly_graded_a}
      \end{align}
    \item 
      The sequence 
      \( k \in \ic{1,d} \mapsto \TripleNorm{\primal}_{k} \)
      is nonincreasing and we have the implication (resp. equivalence), for any \( l\in\ic{1,d} \), 
      \begin{align}
        \lzero\np{\primal} \leq l 
        & \implies
        \TripleNorm{{\primal}}_{l} = 
        \TripleNorm{{\primal}}_{d}
        \eqfinv
        \label{eq:decreasingly_graded_b}
 \\
\left( \right. \text{resp.} \qquad      
    \lzero\np{\primal} \leq l 
       & \iff 
        \TripleNorm{{\primal}}_{l} = 
        \TripleNorm{{\primal}}_{d}
        \quad \bp{       \iff 
          \TripleNorm{{\primal}}_{l} \leq 
          \TripleNorm{{\primal}}_{d} }
        \eqfinp
             \left. \right)
     \label{eq:strictly_decreasingly_graded_b}
      \end{align}
    \item 
      The sequence 
      \( k \in \ic{1,d} \mapsto \TripleNorm{\primal}_{k} \)
      is nonincreasing and we have the inequality  (resp. equality)
      \begin{align}
        \lzero\np{\primal} 
        &\geq 
\min \bset{k \in \ic{1,d} }%
        { \TripleNorm{{\primal}}_{k} = \TripleNorm{{\primal}}_{d} }
        \eqfinv 
        \label{eq:decreasingly_graded_c}
\\
\left( \right. \text{resp.} \qquad      
         \lzero\np{\primal} 
        &= 
          \min \bset{k \in \ic{1,d} }%
        { \TripleNorm{{\primal}}_{k} = \TripleNorm{{\primal}}_{d} }
          \eqfinp
        \label{eq:strictly_decreasingly_graded_c}
              \left. \right)
      \end{align}
      %
    \end{subequations}
  \end{enumerate}
  \label{de:decreasingly_graded}
\end{definition}

\section{Coordinate-$k$ norms and dual coordinate-$k$ norms}
\label{Coordinate-k_and_dual_coordinate-k_norms}

In \S~\ref{Background_on_norms}, we provide background on norms.
Then, in \S~\ref{Definition_of_coordinate-k_and_dual_coordinate-k_norms},
we introduce coordinate-$k$ norms and dual coordinate-$k$ norms.

\subsection{Background on norms}
\label{Background_on_norms}

For any norm~$\TripleNorm{\cdot}$ on~$\RR^d$,
we denote the unit sphere~$\TripleNormSphere$ and the unit ball~$\TripleNormBall$ 
by
\begin{equation}
  \TripleNormSphere= 
      \defset{\primal \in \RR^d}{\TripleNorm{\primal} = 1} 
\eqsepv
    \TripleNormBall = 
      \defset{\primal \in \RR^d}{\TripleNorm{\primal} \leq 1} 
      \eqfinp
 \label{eq:triplenorm_unit_sphere}
\end{equation}


\subsubsubsection{Dual norms}

We recall that the expression
\( \TripleNorm{\dual}_\star = 
  \sup_{ \TripleNorm{\primal} \leq 1 } \proscal{\primal}{\dual} \), 
  \( \forall \dual \in \RR^d \),
defines a norm on~$\RR^d$, 
called the \emph{dual norm} \( \TripleNormDual{\cdot} \).
  By definition of the dual norm, we have the inequality
  \begin{equation}
    \proscal{\primal}{\dual} \leq 
    \TripleNorm{\primal} \times \TripleNormDual{\dual} 
    \eqsepv \forall  \np{\primal,\dual} \in \RR^d \times \RR^d 
    \eqfinp 
    \label{eq:norm_dual_norm_inequality}
  \end{equation}
We denote the unit sphere~$\TripleNormDualSphere$ and the unit ball~$\TripleNormDualBall$
of the dual norm~$\TripleNormDual{\cdot}$ by 
\begin{equation}
      \TripleNormDualSphere = 
      \defset{\dual \in \RR^d}{\TripleNormDual{\dual} = 1} 
      \eqsepv
    \TripleNormDualBall = 
      \defset{\dual \in \RR^d}{\TripleNormDual{\dual} \leq 1} 
      \eqfinp
\label{eq:triplenorm_Dual_unit_sphere}
\end{equation}
%
Denoting by $\sigma_S$ the \emph{support function} of the set~$S \subset \RR^d$
  (\( \sigma_S\np{\dual}=\sup_{\primal \in S} \proscal{\primal}{\dual} \)),
  we have
  \begin{equation}
    \TripleNorm{\cdot} = \sigma_{\TripleNormDualBall} = \sigma_{\TripleNormDualSphere} 
    \mtext{ and } 
    \TripleNormDual{\cdot} = \sigma_{\TripleNormBall} = \sigma_{\TripleNormSphere}
    \eqfinv
    \label{eq:norm_dual_norm}
  \end{equation}
  where \( \TripleNormDualBall =\TripleNormBall^{\odot} 
    = \defset{\dual \in \RR^d}{\proscal{\primal}{\dual} \leq 1 
      \eqsepv \forall \primal \in \TripleNormBall } \)
  is the polar set~\( \TripleNormBall^{\odot} \) 
  of the unit ball~\( \TripleNormBall \).

\subsubsubsection{Restriction norms}

\begin{definition}
  For any norm~$\TripleNorm{\cdot}$ on~$\RR^d$
  and any subset \( K \subset\ic{1,d} \),
  we define 
  \begin{itemize}
  \item 
    the \emph{$K$-restriction norm} \( \TripleNorm{\cdot}_{K} \)
    on the subspace~\( \FlatRR_{K} \) of~\( \RR^d \),
    as defined in~\eqref{eq:FlatRR}, by
    \begin{equation}
      \TripleNorm{\primal}_{K} = \TripleNorm{\primal} 
      \eqsepv
      \forall \primal \in \FlatRR_{K} 
      \eqfinp 
      \label{eq:K_norm}
    \end{equation}
  \item 
    the \emph{$\SetStar{K}$-norm} \( \TripleNorm{\cdot}_{K,\star} \),
    on the subspace~\( \FlatRR_{K} \) of~\( \RR^d \), which is 
    the norm \( \bp{\TripleNorm{\cdot}_{K}}_{\star} \),
    given by the dual norm (on the subspace~\( \FlatRR_{K} \))
    of the restriction norm~\( \TripleNorm{\cdot}_{K} \) 
    to the subspace~\( \FlatRR_{K} \) (first restriction, then dual).
  \end{itemize}
  \label{de:K_norm}
\end{definition}

We have that \cite[Equation~\eqref{OSM-eq:K_star=sigma}]{Chancelier-DeLara:2020_OSM}
\begin{equation}
  \TripleNorm{\dual}_{K,\star}
  =
  \sigma_{ \FlatRR_{K} \cap \TripleNormBall }\np{\dual}  
  = \sigma_{ \FlatRR_{K} \cap \TripleNormSphere }\np{\dual}  
  \eqsepv \forall \dual \in \FlatRR_{K} 
  \eqfinp
  \label{eq:K_star=sigma}
\end{equation}

\subsection{Coordinate-$k$ and dual coordinate-$k$ norms}
\label{Definition_of_coordinate-k_and_dual_coordinate-k_norms}

\subsubsubsection{Source norm}

Let $\TripleNorm{\cdot}$ be a norm on~$\RR^d$, 
that we will call the \emph{source norm}.

\subsubsubsection{Definition of coordinate-$k$ and dual coordinate-$k$ norms}

\begin{definition}
  For \( k \in \ic{1,d} \), we call \emph{coordinate-$k$ norm}
  the norm \( \CoordinateNorm{\TripleNorm{\cdot}}{k} \) 
  whose dual norm is the 
  \emph{dual coordinate-$k$ norm}, denoted by
  \( \CoordinateNormDual{\TripleNorm{\cdot}}{k} \), 
  with expression\footnote{%
The notation \( \sup_{\cardinal{K} \leq k} \) is a shorthand for 
  \( \sup_{ { K \subset \ic{1,d}, \cardinal{K} \leq k}} \).
}
  \begin{equation}
    \CoordinateNormDual{\TripleNorm{\dual}}{k}
    =
    \sup_{\cardinal{K} \leq k} \TripleNorm{\dual_K}_{K,\star} 
    \eqsepv \forall \dual \in \RR^d 
    \eqfinv
    \label{eq:dual_coordinate_norm_definition}
  \end{equation}
  where the $\SetStar{K}$-norm \( \TripleNorm{\cdot}_{K,\star} \) is given in
  Definition~\ref{de:K_norm}.
  \label{de:coordinate_norm}
\end{definition}
It is easily verified that  \( \CoordinateNormDual{\TripleNorm{\cdot}}{k} \)
indeed is a norm.
We adopt the convention \( \CoordinateNormDual{\TripleNorm{\cdot}}{0} = 0 \) 
(although this is not a norm on~$\RR^d$, but a seminorm).

\subsubsubsection{Examples}

Table~\ref{tab:Examples} provides examples
\cite{Chancelier-DeLara:2020_OSM,Chancelier-DeLara:2020_Variational}.
With this, we define the top $(k,q)$-norms in the last right column of 
Table~\ref{tab:Examples}.
The $(p,k)$-support norm, in the middle column of 
Table~\ref{tab:Examples}, is defined as the dual norm of the top $(k,q)$-norm,
with \( 1/p + 1/q =1 \).

\begin{table}
  \centering
  \begin{tabular}{||c||c|c||}
    \hline\hline 
    source norm \( \TripleNorm{\cdot} \) 
    & \( \CoordinateNorm{\TripleNorm{\cdot}}{k} \)
    & \( \CoordinateNormDual{\TripleNorm{\cdot}}{k} \)
    \\
    \hline\hline 
    \( \norm{\cdot}_{p} \) & $(p,k)$-support norm & top $(k,q)$-norm 
    \\
    & \( \Norm{\primal}_{p,k}^{\mathrm{sn}} \)
    & \( \Norm{\dual}_{k,q}^{\mathrm{tn}} \)
    \\
    && \( =\bp{ \sum_{\LocalIndex=1}^{k} \module{ \dual_{\nu(\LocalIndex)} }^q }^{1/q} \), \( 1/p + 1/q =1 \)
    \\
    \hline
    \( \norm{\cdot}_{1} \) 
    & $(1,k)$-support norm 
    & top $(k,\infty)$-norm  
    \\
    & $\ell_{1}$-norm & 
                        $\ell_{\infty}$-norm 
    \\
    & \( \Norm{\primal}_{1,k}^{\mathrm{sn}} = \norm{\primal}_{1} \) 
    & \( \Norm{\dual}_{k,\infty}^{\mathrm{tn}} = \module{ \dual_{\nu(1)} } = \norm{\dual}_{\infty} \) 
    \\
    \hline 
    \( \norm{\cdot}_{2} \) 
    & $(2,k)$-support norm 
    & top $(k,2)$-norm  
    \\
    & & 
        \( \Norm{\dual}_{k,2}^{\mathrm{tn}} = \sqrt{ \sum_{\LocalIndex=1}^{k} \module{ \dual_{\nu(\LocalIndex)} }^2 } \)
    \\
    \hline 
    \( \norm{\cdot}_{\infty} \) & $(\infty,k)$-support norm 
    & top $(k,1)$-norm  
    \\
    & & \( \Norm{\dual}_{k,1}^{\mathrm{tn}} = \sum_{\LocalIndex=1}^{k} \module{ \dual_{\nu(\LocalIndex)} } \)
    \\
    \hline\hline
  \end{tabular}
  \caption{Examples of coordinate-$k$ and dual coordinate-$k$ norms
    generated by the $\ell_p$ source norms 
    \( \TripleNorm{\cdot} = \norm{\cdot}_{p} \) for $p\in [1,\infty]$.
For \( \dual \in \RR^d \), $\nu$ denotes a permutation of \( \{1,\ldots,d\} \) such that
\( \module{ \dual_{\nu(1)} } \geq \module{ \dual_{\nu(2)} } 
\geq \cdots \geq \module{ \dual_{\nu(d)} } \).
    \label{tab:Examples}}
\end{table}
\bigskip

To prepare Sect.~\ref{The_Capra_conjugacy_and_the_lzeropseudonorm},
we provide properties of coordinate-$k$ and dual coordinate-$k$ norms.

\subsubsubsection{Properties of dual coordinate-$k$ norms}

We denote the unit sphere~\( \CoordinateNormDual{\TripleNormSphere}{k} \) 
and the unit ball ~\( \CoordinateNormDual{\TripleNormBall}{k} \)
of the dual coordinate-$k$ norm
\( \CoordinateNormDual{\TripleNorm{\cdot}}{k} \) 
in Definition~\ref{de:coordinate_norm} by
\begin{equation}
     \CoordinateNormDual{\TripleNormSphere}{k} = 
      \defset{\dual \in \RR^d}{\CoordinateNormDual{\TripleNorm{\dual}}{k} = 1} 
      \eqsepv 
    \CoordinateNormDual{\TripleNormBall}{k} = 
      \defset{\dual \in \RR^d}{\CoordinateNormDual{\TripleNorm{\dual}}{k} \leq 1} 
      \eqsepv k\in\ic{1,d}
      \eqfinp
\label{eq:dual_coordinate_norm_unit_sphere_ball}
\end{equation}

\begin{proposition}
  \quad 
  \begin{itemize}
  \item 
    For \( k \in \ic{1,d} \), the dual coordinate-$k$ norm satisfies
    \begin{equation}
      \CoordinateNormDual{\TripleNorm{\dual}}{k}
      =
      \sup_{\cardinal{K} \leq k} 
      \sigma_{ \np{ \FlatRR_{K} \cap \TripleNormSphere } }\np{\dual} 
      =
      \sigma_{ \LevelSet{\lzero}{k} \cap \TripleNormSphere }\np{\dual} 
      =
      \sigma_{ \LevelCurve{\lzero}{k} \cap \TripleNormSphere }\np{\dual} 
      \eqsepv \forall \dual \in \RR^d 
      \eqfinp
      \label{eq:dual_coordinate_norm}
    \end{equation}
  \item 
    We have the equality
    \begin{equation}
      \TripleNormDual{\cdot} 
      =
      \CoordinateNormDual{\TripleNorm{\cdot}}{d}
      \eqfinp  
      \label{eq:dual_coordinate_norm-d_equality}
    \end{equation}
  \item 
    The sequence 
    \( \sequence{\CoordinateNormDual{\TripleNorm{\cdot}}{\LocalIndex}}{\LocalIndex\in\ic{1,d}} \)
    of dual coordinate-$k$ norms in Definition~\ref{de:coordinate_norm}
    is nondecreasing, 
    that is, the following inequalities and equality hold true:
    \begin{equation}
      \CoordinateNormDual{\TripleNorm{\dual}}{1} \leq \cdots \leq 
      \CoordinateNormDual{\TripleNorm{\dual}}{\LocalIndex} \leq 
      \CoordinateNormDual{\TripleNorm{\dual}}{\LocalIndex+1} \leq \cdots \leq 
      \CoordinateNormDual{\TripleNorm{\dual}}{d} 
      = \TripleNormDual{\dual} 
      \eqsepv \forall \dual \in \RR^d
      \eqfinp
      \label{eq:dual_coordinate_norm_inequalities}
    \end{equation}
  \item 
    The sequence 
    \( \sequence{ \CoordinateNormDual{\TripleNormBall}{\LocalIndex}}{\LocalIndex\in\ic{1,d}} \)
    of units balls of the dual coordinate-$k$ norms 
    in Definition~\ref{de:coordinate_norm} 
    is nonincreasing, 
    that is, the following equality and inclusions hold true:
    \begin{equation}
      \TripleNormDualBall=
      \CoordinateNormDual{\TripleNormBall}{d} 
      \subset \cdots \subset \CoordinateNormDual{\TripleNormBall}{\LocalIndex+1}
      \subset \CoordinateNormDual{\TripleNormBall}{\LocalIndex} \subset \cdots 
      \subset \CoordinateNormDual{\TripleNormBall}{1}
      \eqfinp 
      \label{eq:dual_coordinate_norm_unit-balls_inclusions}
    \end{equation}
  \end{itemize}
\end{proposition}

\begin{proof}
  
  \noindent $\bullet$ 
  For any \( \dual \in \RR^d \), we have
  \begin{align*}
    \CoordinateNormDual{\TripleNorm{\dual}}{k}
    &=
      \sup_{\cardinal{K} \leq k} \TripleNorm{\dual_K}_{K,\star} 
      \tag{by definition~\eqref{eq:dual_coordinate_norm_definition}
      of \( \CoordinateNormDual{\TripleNorm{\dual}}{k} \)}
    \\
    &=
      \sup_{\cardinal{K} \leq k} 
      \sigma_{ \np{ \FlatRR_{K} \cap \TripleNormSphere } }\np{\dual_K}
      \tag{as \( \TripleNorm{\dual_K}_{K,\star} =
      \sigma_{ \np{ \FlatRR_{K} \cap \TripleNormSphere } }\np{\dual_K} \) 
      by~\eqref{eq:K_star=sigma}}
    \\
    &= 
      \sup_{\cardinal{K} \leq k} 
      \sup_{\primal \in \FlatRR_{K} \cap \TripleNormSphere } \proscal{\primal}{\dual_K} 
      \tag{by definition of the support function
      \( \sigma_{ \np{ \FlatRR_{K} \cap \TripleNormSphere } } \)}
    \\
    &= 
      \sup_{\cardinal{K} \leq k} 
      \sup_{\primal \in \FlatRR_{K} \cap \TripleNormSphere } \proscal{\primal}{\dual} 
      \tag{by~\eqref{eq:orthogonal_projection_self-dual} as \( \primal \in \FlatRR_{K} \)}
    \\
    &= 
      \sup_{\cardinal{K} \leq k} 
      \sigma_{ \np{ \FlatRR_{K} \cap \TripleNormSphere } }\np{\dual}
      \tag{by definition of the support function
      \( \sigma_{ \np{ \FlatRR_{K} \cap \TripleNormSphere } } \)}
    \\
    &=
      \sigma_{ \bigcup_{ \cardinal{K} \leq k} 
      \np{ \FlatRR_{K} \cap \TripleNormSphere } }\np{\dual}
      \tag{as the support function turns a union of sets into a supremum}  
    \\
    &=
      \sigma_{ \LevelSet{\lzero}{k} \cap \TripleNormSphere }\np{\dual} 
      \tag{as \( \LevelSet{\lzero}{k} \cap \TripleNormSphere 
      = \bigcup_{ {\cardinal{K} \leq k}} \np{ \FlatRR_{K} \cap \TripleNormSphere }
      \) by~\eqref{eq:level_set_pseudonormlzero}}
      \eqfinp
  \end{align*}
  To finish, we will now prove that 
  \( \sigma_{ \LevelSet{\lzero}{k} \cap \TripleNormSphere } =
  \sigma_{ \LevelCurve{\lzero}{k} \cap \TripleNormSphere } \).
  For this purpose, we show in two steps that 
  \(      \LevelSet{\lzero}{k} \cap \TripleNormSphere 
  =
  \overline{ \LevelCurve{\lzero}{k} \cap \TripleNormSphere } \).

  First, we establish the (known) fact that 
  \( \overline{ \LevelCurve{\lzero}{k} } = \LevelSet{\lzero}{k} \). 
  The inclusion \( \overline{ \LevelCurve{\lzero}{k} } 
  \subset \LevelSet{\lzero}{k} \) is easy because,
  on the one hand, \( \LevelCurve{\lzero}{k} \subset \LevelSet{\lzero}{k} \)
  and, on the other hand, 
  the level set \( \LevelSet{\lzero}{k} \)
  in~\eqref{eq:pseudonormlzero_level_set} is closed,
  as follows from the well-known property that
  the pseudonorm~$\lzero$ is lower semicontinuous.
  There remains to prove the reverse inclusion
  \( \LevelSet{\lzero}{k} \subset \overline{ \LevelCurve{\lzero}{k} } \).
  For this purpose, we consider
  \( \primal \in \LevelSet{\lzero}{k} \). 
  If \( \primal \in \LevelCurve{\lzero}{k} \), obviously 
  \( \primal \in \overline{ \LevelCurve{\lzero}{k} } \).
  Therefore, we suppose that \( \lzero\np{\primal}=l < k \).
  By definition of \( \lzero\np{\primal} \) in~\eqref{eq:pseudo_norm_l0}, there exists 
  \( L \subset \ba{1,\ldots,d} \) such that 
  \( \cardinal{L}=l < k \) and \( \primal = \primal_L \).
  For \( \epsilon > 0 \), define \( \primal^\epsilon \) as
  coinciding with  \( \primal \) except for 
  $k-l$ indices outside~$L$ for which the components are 
  \( \epsilon > 0 \).
  By construction \( \lzero\np{\primal^\epsilon}=k \) and
  \( \primal^\epsilon \to \primal \) when \( \epsilon \to 0 \).
  This proves that 
  \( \LevelSet{\lzero}{k} \subset \overline{ \LevelCurve{\lzero}{k} } \).

  Second, we prove that \( \LevelSet{\lzero}{k} \cap \TripleNormSphere 
  = \overline{ \LevelCurve{\lzero}{k} \cap \TripleNormSphere } \).
  The inclusion 
  \( \overline{ \LevelCurve{\lzero}{k} \cap \TripleNormSphere } 
  \subset \LevelSet{\lzero}{k} \cap \TripleNormSphere \),
  is easy.
  Indeed, 
  \( \overline{ \LevelCurve{\lzero}{k} } = \LevelSet{\lzero}{k} 
  \implies
  \overline{ \LevelCurve{\lzero}{k} \cap \TripleNormSphere } \subset 
  \overline{ \LevelCurve{\lzero}{k} } \cap \overline{ \TripleNormSphere } = 
  \LevelSet{\lzero}{k} \cap \TripleNormSphere \).
  To prove the reverse inclusion
  \( \LevelSet{\lzero}{k} \cap \TripleNormSphere 
  \subset \overline{ \LevelCurve{\lzero}{k} \cap \TripleNormSphere } \),
  we consider \( \primal \in \LevelSet{\lzero}{k} \cap \TripleNormSphere \).
  As we have just seen that \( \LevelSet{\lzero}{k} = 
  \overline{ \LevelCurve{\lzero}{k} }\), 
  we deduce that \( \primal \in \overline{ \LevelCurve{\lzero}{k} }\).
  Therefore, there exists a sequence
  \( \sequence{z_n}{n\in\NN} \) in \( \LevelCurve{\lzero}{k} \)
  such that \( z_n \to \primal \) when \( n \to +\infty \).
  Since \( \primal \in \TripleNormSphere \), we can always suppose that 
  \( z_n \neq 0 \), for all $n\in\NN$. Therefore \( z_n/\TripleNorm{z_n} \) is well
  defined and, when \( n \to +\infty \), 
  we have \( z_n/\TripleNorm{z_n} \to \primal/\TripleNorm{\primal}=\primal \)
  since \( \primal \in \TripleNormSphere = \defset{\primal \in \PRIMAL}{\TripleNorm{\primal} = 1} \).
  Now, on the one hand, 
  \( z_n/\TripleNorm{z_n} \in \LevelCurve{\lzero}{k} \), for all $n\in\NN$,
  and, on the other hand, \( z_n/\TripleNorm{z_n} \in \TripleNormSphere \).
  As a consequence \( z_n/\TripleNorm{z_n} \in \LevelCurve{\lzero}{k} \cap \TripleNormSphere \),
  and we conclude that \( \primal \in 
  \overline{ \LevelCurve{\lzero}{k} \cap \TripleNormSphere } \). 
  Thus, we have proved that 
  \( \LevelSet{\lzero}{k} \cap \TripleNormSphere 
  \subset \overline{ \LevelCurve{\lzero}{k} \cap \TripleNormSphere } \).
  \medskip

  From \( \LevelSet{\lzero}{k} \cap \TripleNormSphere 
  = \overline{ \LevelCurve{\lzero}{k} \cap \TripleNormSphere } \), we get that 
  \( \sigma_{ \LevelSet{\lzero}{k} \cap \TripleNormSphere } =
  \sigma_{ \overline{ \LevelCurve{\lzero}{k} \cap \TripleNormSphere } } = 
  \sigma_{ \LevelCurve{\lzero}{k} \cap \TripleNormSphere } \),
  by \cite[Proposition~7.13]{Bauschke-Combettes:2017}.
  Thus, we have proved all equalities in~\eqref{eq:dual_coordinate_norm}.
  \medskip

  \noindent $\bullet$ 
  By the equality 
  \( \CoordinateNormDual{\TripleNorm{\dual}}{k}
  = \sigma_{ \LevelSet{\lzero}{k} \cap \TripleNormSphere }\np{\dual} \)
  in~\eqref{eq:dual_coordinate_norm},
  we get that, for all $\dual \in \RR^d$,
  \(         \CoordinateNormDual{\TripleNorm{\dual}}{d}
  =
  \sigma_{ \LevelSet{\lzero}{d} \cap \TripleNormSphere }\np{\dual}
  = \sigma_{\TripleNormSphere}\np{\dual}
  = \TripleNormDual{\dual} \)
  since \( \LevelSet{\lzero}{d} = \RR^d \) and by~\eqref{eq:norm_dual_norm}.
  \medskip

  \noindent $\bullet$ 
  The inequalities in~\eqref{eq:dual_coordinate_norm_inequalities}
  easily derive from the very definition~\eqref{eq:dual_coordinate_norm_definition}
  of the dual coordinate-$k$ norms~\( \CoordinateNormDual{\TripleNorm{\cdot}}{k}
  \).
  The last equality is just the equality~\eqref{eq:dual_coordinate_norm-d_equality}.
  \medskip

  \noindent $\bullet$ 
  The equality and the inclusions in~\eqref{eq:dual_coordinate_norm_unit-balls_inclusions}
  directly follow from the inequalities and the equality between norms
  in~\eqref{eq:dual_coordinate_norm_inequalities}.
  \medskip

  This ends the proof. 
\end{proof}

\subsubsubsection{Properties of coordinate-$k$ norms}

We denote the unit sphere~\( \CoordinateNorm{\TripleNormSphere}{k} \) 
and the unit ball~\( \CoordinateNorm{\TripleNormBall}{k} \) 
of the coordinate-$k$ 
norm~\( \CoordinateNorm{\TripleNorm{\cdot}}{k} \) by
\begin{equation}
  \CoordinateNorm{\TripleNormSphere}{k} =
      \defset{\primal \in \RR^d}{%
      \CoordinateNorm{\TripleNorm{\primal}}{k} = 1 } 
      \eqsepv
    \CoordinateNorm{\TripleNormBall}{k} =
      \defset{\primal \in \RR^d}{%
      \CoordinateNorm{\TripleNorm{\primal}}{k} \leq 1 } 
      \eqfinp
\label{eq:coordinate_norm_unit_sphere_ball}
\end{equation}
We adopt the convention 
\( \CoordinateNorm{\TripleNormBall}{0} =\{0\} \)
(although this is not the unit ball of a norm on~$\RR^d$).
%

\begin{proposition}
  \quad
  \begin{itemize}
  \item 
    For \( k \in \ic{1,d} \), the coordinate-$k$ norm
    \( \CoordinateNorm{\TripleNorm{\cdot}}{k} \) 
    has unit ball
    \begin{equation}
      \CoordinateNorm{\TripleNormBall}{k}=
      \closedconvexhull\bp{ \bigcup_{ \cardinal{K} \leq k} 
        \np{ \FlatRR_{K} \cap \TripleNormSphere } } 
      \eqfinv
      \label{eq:coordinate_norm_unit_ball_property}
    \end{equation}
    where \( \closedconvexhull\np{S} \) denotes the closed convex hull
    of a subset \( S \subset \RR^d \).
  \item 
    We have the equality
    \begin{equation}
      \CoordinateNorm{\TripleNorm{\cdot}}{d}
      =
      \TripleNorm{\cdot}
      \eqfinp  
      \label{eq:coordinate_norm-d_equality}
    \end{equation}
  \item 
    The sequence 
    \( \sequence{\CoordinateNorm{\TripleNorm{\cdot}}{\LocalIndex}}{\LocalIndex\in\ic{1,d}} \)
    of coordinate-$k$ norms in Definition~\ref{de:coordinate_norm} is nonincreasing, 
    that is, the following equality and inequalities hold true:
    \begin{equation}
      \TripleNorm{\primal}=
      \CoordinateNorm{\TripleNorm{\primal}}{d}
      \leq \cdots \leq 
      \CoordinateNorm{\TripleNorm{\primal}}{\LocalIndex+1}\leq 
      \CoordinateNorm{\TripleNorm{\primal}}{\LocalIndex}
      \leq \cdots \leq 
      \CoordinateNorm{\TripleNorm{\primal}}{1} 
      \eqsepv \forall \primal \in \RR^d 
      \eqfinp
      \label{eq:coordinate_norm_inequalities}
    \end{equation}
  \item 
    The sequence 
    \( \sequence{ \CoordinateNorm{\TripleNormBall}{\LocalIndex}}{\LocalIndex\in\ic{1,d}} \)
    of units balls of the coordinate-$k$ norms 
    in~\eqref{eq:coordinate_norm_unit_ball_property} is nondecreasing, 
    that is, the following inclusions and equality hold true:
    \begin{equation}
      \CoordinateNorm{\TripleNormBall}{1} 
      \subset \cdots \subset
      \CoordinateNorm{\TripleNormBall}{\LocalIndex} \subset \CoordinateNorm{\TripleNormBall}{\LocalIndex+1} 
      \subset \cdots \subset \CoordinateNorm{\TripleNormBall}{d} 
      = \TripleNormBall
      \eqfinp 
      \label{eq:coordinate_norm_unit-balls_inclusions}
    \end{equation}
  \end{itemize}
  \label{pr:coordinate_norm}
\end{proposition}

\begin{proof}

  \noindent $\bullet$ 
  For any \( \dual \in \RR^d \), we have
  \begin{align*}
    \CoordinateNormDual{\TripleNorm{\dual}}{k}
    &=
      \sup_{\cardinal{K} \leq k} 
      \sigma_{ \np{ \FlatRR_{K} \cap \TripleNormSphere } }\np{\dual}
      \tag{by~\eqref{eq:dual_coordinate_norm}}
    \\
    &=
      \sigma_{ \bigcup_{ \cardinal{K} \leq k} 
      \np{ \FlatRR_{K} \cap \TripleNormSphere } }\np{\dual}
      \tag{as the support function turns a union of sets into a supremum}  
    \\
    &=
      \sigma_{ \closedconvexhull\bp{ \bigcup_{ \cardinal{K} \leq k} 
      \np{ \FlatRR_{K} \cap \TripleNormSphere } } }\np{\dual}
      \tag{by \cite[Proposition~7.13]{Bauschke-Combettes:2017} }
  \end{align*}
  and we conclude that 
  \( \CoordinateNorm{\TripleNormBall}{k} =
  \closedconvexhull\bp{ \bigcup_{ \cardinal{K} \leq k} 
    \np{ \FlatRR_{K} \cap \TripleNormSphere } } \) 
  by~\eqref{eq:norm_dual_norm}.
  Thus, we have proved~\eqref{eq:coordinate_norm_unit_ball_property}.
  \medskip

  \noindent $\bullet$ 
  From the equality~\eqref{eq:dual_coordinate_norm-d_equality},
  we deduce the equality~\eqref{eq:coordinate_norm-d_equality} between the dual norms
by definition of the dual norm.
  \medskip

  \noindent $\bullet$ 
  The equality and inequalities between norms in~\eqref{eq:coordinate_norm_inequalities}
  easily derive from the inclusions and equality between unit balls 
  in~\eqref{eq:coordinate_norm_unit-balls_inclusions}.
  \medskip

  \noindent $\bullet$ 
  The inclusions and equality between unit balls 
  in~\eqref{eq:coordinate_norm_unit-balls_inclusions}
  directly follow from the inclusions and equality between unit balls 
  in~\eqref{eq:dual_coordinate_norm_unit-balls_inclusions}
  and from \( \CoordinateNorm{\TripleNormBall}{\LocalIndex} =
  \bp{  \CoordinateNormDual{\TripleNormBall}{\LocalIndex} }^{\odot} \),
  the polar set of~\(  \CoordinateNormDual{\TripleNormBall}{\LocalIndex} \).
  \medskip

  This ends the proof. 
\end{proof}

We recall that the normed space 
\( \bp{\RR^d,\TripleNorm{\cdot}} \) is said to be \emph{strictly convex}
if the unit ball~$\TripleNormBall$ 
(of the norm~$\TripleNorm{\cdot}$) is \emph{rotund}, that is,
if all points of the unit sphere~$\TripleNormSphere$ are extreme points
of the unit ball~$\TripleNormBall$.
The normed space \( \bp{\RR^d,\norm{\cdot}_{p}} \), equipped with 
the $\ell_p$-norm~$\norm{\cdot}_{p}$ (for $p\in [1,\infty]$), is strictly
convex if and only if $p\in ]1,\infty[$.

We now show that the sequences 
\( \sequence{\CoordinateNorm{\TripleNorm{\cdot}}{\LocalIndex}}{\LocalIndex\in\ic{1,d}} \)
of coordinate-$k$ norms (in Definition~\ref{de:coordinate_norm}) 
are naturally decreasingly graded with respect to the \lzeropseudonorm\
(as in Definition~\ref{de:decreasingly_graded}).
Part of the proof relies upon the forthcoming
Lemma~\ref{lem:level_set_pseudonormlzero_intersection_sphere_rotund}. 

\begin{proposition} 
  \label{pr:coordinate_norm_graded}
  \quad
  \begin{enumerate}
  \item
    \label{it:coordinate_norm_graded}
    The nonincreasing sequence 
    \( \sequence{\CoordinateNorm{\TripleNorm{\cdot}}{\LocalIndex}}{\LocalIndex\in\ic{1,d}} \)
    of coordinate-$k$ norms 
    is decreasingly graded with respect to the \lzeropseudonorm, that is,
    for any \( l\in\ic{1,d} \), 
    \begin{subequations}
      \begin{equation}
        \lzero\np{\primal} \leq l \implies 
        \TripleNorm{\primal}=
        \CoordinateNorm{\TripleNorm{\primal}}{l} 
        \eqsepv \forall \primal \in \RR^d 
        \eqfinp
        \label{eq:coordinate_norm_graded_b}
      \end{equation}
    \item
      \label{it:coordinate_norm_strictly_graded}
      If the normed space 
      \( \bp{\RR^d,\TripleNorm{\cdot}} \) is strictly convex, then 
      the nonincreasing sequence 
      \( \sequence{\CoordinateNorm{\TripleNorm{\cdot}}{\LocalIndex}}{\LocalIndex\in\ic{1,d}} \)
      of coordinate-$k$ norms 
      is strictly decreasingly graded with respect to the \lzeropseudonorm, that
      is, for any \( l\in\ic{1,d} \), 
      \begin{equation}
        \lzero\np{\primal} \leq l
        \iff 
        \TripleNorm{\primal}=
        \CoordinateNorm{\TripleNorm{\primal}}{l} 
        \eqsepv \forall \primal \in \RR^d 
        \eqfinp
        \label{eq:coordinate_norm_strictly_graded_b}
      \end{equation}
    \end{subequations}
  \end{enumerate}
\end{proposition}

\begin{proof}
  \quad 

  \noindent $\bullet$ 
  We prove Item~\ref{it:coordinate_norm_graded}.
  As the sequence 
  \( \sequence{\CoordinateNorm{\TripleNorm{\cdot}}{\LocalIndex}}{\LocalIndex\in\ic{1,d}} \)
  of coordinate-$k$ norms is nonincreasing
  by~\eqref{eq:dual_coordinate_norm_inequalities},
  it suffices to show that 
  \eqref{eq:decreasingly_graded_b} holds true --- that is,
  that \eqref{eq:coordinate_norm_graded_b} holds true ---
  to prove that the sequence is decreasingly graded with respect to the \lzeropseudonorm\
  (see Definition~\ref{de:decreasingly_graded}).

  Now,
  for any \( \primal \in \RR^d \) and
  for any \( k \in \ic{1,d} \), we have\footnote{%
In what follows, by ``or'' we mean the so-called \emph{exclusive or}
    (exclusive disjunction). Thus, every ``or'' should be understood as ``or
    $\primal\not=0$ and'' \label{ft:exclusive_or}}
  
  \begin{align*}
    \primal \in \LevelSet{\lzero}{k} 
    &\iff 
      \primal=0 \text{ or }
      \frac{\primal}{\TripleNorm{\primal}} \in \LevelSet{\lzero}{k}
      \tag{by 0-homogeneity~\eqref{eq:lzeropseudonorm_is_0-homogeneous}
      of the \lzeropseudonorm, and
      by definition~\eqref{eq:pseudonormlzero_level_set}
      of $\LevelSet{\lzero}{k}$ }
    \\
    &\iff 
      \primal=0 \text{ or }
      \frac{\primal}{\TripleNorm{\primal}} \in 
      \LevelSet{\lzero}{k} \cap \TripleNormSphere 
      \tag{as \( \frac{\primal}{\TripleNorm{\primal}} \in \TripleNormSphere \)
      by definition~\eqref{eq:triplenorm_unit_sphere} of the unit sphere~$\TripleNormSphere$ }
    \\
    &\iff 
      \primal=0 \text{ or }
      \frac{\primal}{\TripleNorm{\primal}} \in 
      \bigcup_{ {\cardinal{K} \leq k}} \np{ \FlatRR_{K} \cap \TripleNormSphere }
      \tag{as \( \LevelSet{\lzero}{k} 
      = \bigcup_{ {\cardinal{K} \leq k} } \FlatRR_{K} \) by~\eqref{eq:level_set_pseudonormlzero}}
    \\
    &\implies
      \primal=0 \text{ or }
      \frac{\primal}{\TripleNorm{\primal}} \in 
      \CoordinateNorm{\TripleNormBall}{k} 
      \tag{as \( \CoordinateNorm{\TripleNormBall}{k}=
      \closedconvexhull\bp{ \bigcup_{ \cardinal{K} \leq k} 
      \np{ \FlatRR_{K} \cap \TripleNormSphere } } \) by~\eqref{eq:coordinate_norm_unit_ball_property}}
    \\
    &\implies
      \primal=0 \text{ or }
      \CoordinateNorm{\TripleNorm{\frac{\primal}{\TripleNorm{\primal}}}}{k}
      \leq 1 
      \tag{since \( \CoordinateNorm{\TripleNormBall}{k} \) is the unit ball
      of the norm \( \CoordinateNorm{\TripleNorm{\cdot}}{k} \) 
      by~\eqref{eq:coordinate_norm_unit_sphere_ball} }
    \\
    &\implies
      \CoordinateNorm{\TripleNorm{\primal}}{k}
      \leq \TripleNorm{\primal}
    \\
    &\implies
      \CoordinateNorm{\TripleNorm{\primal}}{k}
      \leq \TripleNorm{\primal}
      =\CoordinateNorm{\TripleNorm{\primal}}{d}
      \tag{where the last equality comes
      from~\eqref{eq:coordinate_norm_inequalities} } 
    \\
    &\implies
      \CoordinateNorm{\TripleNorm{\primal}}{k}
      = \CoordinateNorm{\TripleNorm{\primal}}{d}
      \tag{as \( \CoordinateNorm{\TripleNorm{\primal}}{k}
      \geq \CoordinateNorm{\TripleNorm{\primal}}{d} \) by~\eqref{eq:coordinate_norm_inequalities}} 
      \eqfinp
  \end{align*}
  Therefore, we have obtained~\eqref{eq:coordinate_norm_graded_b}.
  \medskip

  \noindent $\bullet$ 
  We prove Item~\ref{it:coordinate_norm_strictly_graded}.
  As the sequence 
  \( \sequence{\CoordinateNorm{\TripleNorm{\cdot}}{\LocalIndex}}{\LocalIndex\in\ic{1,d}} \)
  of coordinate-$k$ norms is nonincreasing
  by~\eqref{eq:dual_coordinate_norm_inequalities},
  it suffices to show that 
  \eqref{eq:strictly_decreasingly_graded_b} holds true --- that is,
  that \eqref{eq:coordinate_norm_strictly_graded_b} holds true ---
  to prove that the sequence is strictly decreasingly graded with respect to the \lzeropseudonorm\
  (see Definition~\ref{de:decreasingly_graded}).

  We suppose that the normed space 
  \( \bp{\RR^d,\TripleNorm{\cdot}} \) is strictly convex. Then, 
  for any \( \primal \in \RR^d \) and
  for any \( k \in \ic{1,d} \), we have \footnote{%
See Footnote~\ref{ft:exclusive_or}.}
  \begin{align*}
    \primal \in \LevelSet{\lzero}{k} 
    &\iff 
      \primal=0 \text{ or }
      \frac{\primal}{\TripleNorm{\primal}} \in \LevelSet{\lzero}{k}
      \intertext{by 0-homogeneity~\eqref{eq:lzeropseudonorm_is_0-homogeneous}
      of the \lzeropseudonorm, and
      by definition~\eqref{eq:pseudonormlzero_level_set}
      of $\LevelSet{\lzero}{k}$ }
    &\iff 
      \primal=0 \text{ or }
      \frac{\primal}{\TripleNorm{\primal}} \in 
      \LevelSet{\lzero}{k} \cap \TripleNormSphere 
      \tag{as \( \frac{\primal}{\TripleNorm{\primal}} \in \TripleNormSphere \)
      by definition~\eqref{eq:triplenorm_unit_sphere} of the unit sphere~$\TripleNormSphere$ }
    \\
    &\iff 
      \primal=0 \text{ or }
      \frac{\primal}{\TripleNorm{\primal}} \in 
      \CoordinateNorm{\TripleNormBall}{k} \cap \TripleNormSphere 
      \intertext{as \( \LevelSet{\lzero}{k} \cap \TripleNormSphere =
      \CoordinateNorm{\TripleNormBall}{k} \cap \TripleNormSphere \) by~\eqref{eq:level_set_l0_inter_sphere_b} 
      since the assumption of
      Lemma~\ref{lem:level_set_pseudonormlzero_intersection_sphere_rotund} 
      is satisfied, that is, the normed space 
      \( \bp{\RR^d,\TripleNorm{\cdot}} \) is strictly convex} 
    &\iff 
      \primal=0 \text{ or }
      \CoordinateNorm{\TripleNorm{\frac{\primal}{\TripleNorm{\primal}}}}{k}
      \leq 1 
      \tag{since \( \CoordinateNorm{\TripleNormBall}{k} \) is the unit ball
      of the norm \( \CoordinateNorm{\TripleNorm{\cdot}}{k} \) 
      by~\eqref{eq:coordinate_norm_unit_sphere_ball} }
    \\
    &\iff 
      \CoordinateNorm{\TripleNorm{\primal}}{k}
      \leq \TripleNorm{\primal}
    \\
    &\iff 
      \CoordinateNorm{\TripleNorm{\primal}}{k}
      \leq \TripleNorm{\primal}
      =\CoordinateNorm{\TripleNorm{\primal}}{d}
      \tag{where the last equality comes
      from~\eqref{eq:coordinate_norm_inequalities} } 
    \\
    &\iff 
      \CoordinateNorm{\TripleNorm{\primal}}{k}
      = \CoordinateNorm{\TripleNorm{\primal}}{d}
      \tag{as \( \CoordinateNorm{\TripleNorm{\primal}}{k}
      \geq \CoordinateNorm{\TripleNorm{\primal}}{d} \) by~\eqref{eq:coordinate_norm_inequalities}} 
      \eqfinp
  \end{align*}
  Therefore, we have obtained~\eqref{eq:coordinate_norm_strictly_graded_b}.
  \medskip

  This ends the proof. 
\end{proof}

\begin{table}[htbp]
  \centering
  \begin{tabular}{||l||c|c||}
    \hline \hline
    & \multicolumn{2}{c||}{ \( \sequence{\CoordinateNorm{\TripleNorm{\cdot}}{\LocalIndex}}{\LocalIndex\in\ic{1,d}} \)}
    \\
    & graded & strictly graded 
    \\ 
    \hline \hline
    \( \TripleNorm{\cdot} \) is any norm & \checkmark &
    \\
    \hline
    \( \np{\RR^d,\TripleNorm{\cdot}} \) 
    is strictly convex && \checkmark 
    \\
    \hline \hline
  \end{tabular}
  \caption{Table of results. It reads as follows:
    to obtain that the sequence \( \sequence{\CoordinateNorm{\TripleNorm{\cdot}}{\LocalIndex}}{\LocalIndex\in\ic{1,d}} \)
    be graded (second column), it suffices that 
    \( \TripleNorm{\cdot} \) be any norm;
    to obtain that the sequence \( \sequence{\CoordinateNorm{\TripleNorm{\cdot}}{\LocalIndex}}{\LocalIndex\in\ic{1,d}} \)
    be strictly graded (third column), it suffices that 
    \( \np{\RR^d,\TripleNorm{\cdot}} \) 
    be strictly convex.
    \label{tab:results}}
\end{table}

Table~\ref{tab:results} summarizes the results of Proposition~\ref{pr:coordinate_norm_graded}.
As an application with any $\ell_p$-norm~$\norm{\cdot}_{p}$ 
for source norm (for $p\in [1,\infty]$),
we obtain that the nonincreasing sequence 
\( \sequence{\Norm{\cdot}_{p,\LocalIndex}^{\mathrm{sn}}}{\LocalIndex\in\ic{1,d}} \)
of $(p,k)$-support norms (see Table~\ref{tab:Examples})
is  strictly decreasingly graded \wrt\ the \lzeropseudonorm\
for $p\in ]1,\infty[$. This gives, by~\eqref{eq:strictly_decreasingly_graded_c}:
\begin{subequations}
  \begin{equation}
    \lzero\np{\primal} 
    = \min \Bset{k \in \ic{1,d} }%
    { \Norm{\primal}_{p,k}^{\mathrm{sn}} = \Norm{\primal}_{p} } 
    \eqsepv \forall \primal \in \RR^d
    \eqsepv \forall p\in ]1,\infty[
    \eqfinp
  \end{equation}
  We also have that the sequence 
  \( \sequence{\Norm{\cdot}_{p,\LocalIndex}^{\mathrm{sn}}}{\LocalIndex\in\ic{1,d}} \)
  is decreasingly graded with respect to the \lzeropseudonorm\
  for $p\in [1,\infty]$. Looking at Table~\ref{tab:Examples}, the only
  interesting case is for \( p=\infty \), giving, by~\eqref{eq:decreasingly_graded_c}:
  \begin{equation}
    \lzero\np{\primal} 
    \geq \min \Bset{k \in \ic{1,d} }%
    { \Norm{\primal}_{\infty,k}^{\mathrm{sn}} = \Norm{\primal}_{\infty} } 
    \eqsepv \forall \primal \in \RR^d
    \eqfinp
  \end{equation}
\end{subequations}

\begin{lemma}  %
  \label{lem:level_set_pseudonormlzero_intersection_sphere_rotund}
  Let $\TripleNorm{\cdot}$ be a norm on~$\RR^d$.
  If the normed space 
  \( \bp{\RR^d,\TripleNorm{\cdot}} \) is strictly convex,
  we have the equality
  \begin{equation}
    \LevelSet{\lzero}{k} \cap \TripleNormSphere 
    =
    \CoordinateNorm{\TripleNormBall}{k} \cap \TripleNormSphere 
    \eqsepv \forall k \in \ic{0,d} 
    \eqfinv
    \label{eq:level_set_l0_inter_sphere_b}
  \end{equation}
  where \( \LevelSet{\lzero}{k} \) is the level set
  in~\eqref{eq:pseudonormlzero_level_set}
  of the \lzeropseudonorm\ in~\eqref{eq:pseudo_norm_l0},
  where $\TripleNormSphere$ is the unit sphere
  in~\eqref{eq:triplenorm_unit_sphere},
  and where \( \CoordinateNorm{\TripleNormBall}{k} \) 
  in~\eqref{eq:coordinate_norm_unit_sphere_ball}
is the unit ball
  of the norm \( \CoordinateNorm{\TripleNorm{\cdot}}{k} \).
\end{lemma}

\begin{proof} 
  It is proved in 
\cite[Proposition~\ref{OSM-pr:level_set_pseudonormlzero_intersection_sphere_rotund}]{Chancelier-DeLara:2020_OSM} that,
  if the unit ball~$\TripleNormBall$ is rotund --- that is, 
  if the normed space 
  \( \bp{\RR^d,\TripleNorm{\cdot}} \) is strictly convex --- 
  and 
  if $A$ is a closed subset of $\TripleNormSphere$,
  then $A = \closedconvexhull(A) \cap \TripleNormSphere$.

  Now, we turn to the proof. 
  First, we observe that the level set \( \LevelSet{\lzero}{k} \) is closed
  because the pseudonorm~$\lzero$ is lower semi continuous.
  Second, we have
  \begin{align*}
    \LevelSet{\lzero}{k} \cap \TripleNormSphere 
    &= 
      \closedconvexhull\bp{\LevelSet{\lzero}{k} \cap \TripleNormSphere} 
      \cap \TripleNormSphere 
      \intertext{because 
      \( \LevelSet{\lzero}{k} \cap \TripleNormSphere \subset \TripleNormSphere \) 
      and is closed, and because the unit ball~$\TripleNormBall$ is rotund }
    &= 
      \closedconvexhull\bp{ \bigcup_{ {\cardinal{K} \leq k}} 
      \np{ \FlatRR_{K} \cap \TripleNormSphere } }
      \cap \TripleNormSphere 
      \tag{by~\eqref{eq:level_set_pseudonormlzero} }
    \\
    &= \CoordinateNorm{\TripleNormBall}{k} \cap \TripleNormSphere 
      \tag{by~\eqref{eq:coordinate_norm_unit_ball_property} }
      \eqfinp
  \end{align*}

  This ends the proof. 
\end{proof}

\section{The \Capra-conjugacy and the \lzeropseudonorm}
\label{The_Capra_conjugacy_and_the_lzeropseudonorm}

We introduce the coupling~\Capra\
in~\S\ref{Constant_along_primal_rays_coupling}.
Then, we provide formulas 
for \Capra-conjugates of functions of the \lzeropseudonorm\
in~\S\ref{CAPRA_conjugates_related_to_the_pseudo_norm},
for \Capra-biconjugates of functions of the \lzeropseudonorm\
in~\S\ref{CAPRA_biconjugates_related_to_the_pseudo_norm},
and 
for \Capra-subdifferentials of functions of the \lzeropseudonorm\
in~\S\ref{Capra-subdifferentials_related_to_the_lzeropseudonorm}.

We work on the Euclidian space~$\RR^d$
(with $d \in \NN^*$), equipped with the scalar product 
\( \proscal{\cdot}{\cdot} \) (but not necessarily with the Euclidian norm).
As we manipulate functions with values in~$\barRR = [-\infty,+\infty] $,
we adopt the Moreau \emph{lower ($\LowPlus$) and upper ($\UppPlus$) additions} \cite{Moreau:1970},
which extend the usual addition~($+$) with 
\( \np{+\infty} \LowPlus \np{-\infty}=\np{-\infty} \LowPlus \np{+\infty}=-\infty \) and
\( \np{+\infty} \UppPlus \np{-\infty}=\np{-\infty} \UppPlus \np{+\infty}=+\infty \).
For any subset \( \Uncertain \subset \UNCERTAIN \) of a set~\( \UNCERTAIN \),
$\delta_{\Uncertain} : \UNCERTAIN \to \barRR $ denotes the \emph{characteristic function} of the
set~$\Uncertain$: 
\( \delta_{\Uncertain}\np{\uncertain} = 0 \) if \( \uncertain \in \Uncertain \),
and \( \delta_{\Uncertain}\np{\uncertain} = +\infty \) 
if \( \uncertain \not\in \Uncertain \).

\subsection{Constant along primal rays coupling (\Capra)}
\label{Constant_along_primal_rays_coupling}

We introduce 
the coupling~\Capra, which is a special case of
one-sided linear coupling, 
as introduced in~\cite{Chancelier-DeLara:2021_ECAPRA_JCA}.
Fenchel-Moreau conjugacies are recalled in Appendix~\ref{Appendix}.

\begin{definition}
  Let $\TripleNorm{\cdot}$ be a norm on~$\RR^d$.
  We define the \emph{constant along primal rays coupling}~$\CouplingCapra$, or \Capra,
  between $\RR^d$ and itself by
  \begin{equation}
    \forall \dual \in \RR^d \eqsepv 
    \CouplingCapra\np{\primal, \dual} =
    \frac{ \proscal{\primal}{\dual} }{ \TripleNorm{\primal} }
      \eqsepv \forall \primal \in \RR^d\backslash\{0\} 
      \mtext{ and }
      \CouplingCapra\np{0, \dual} = 0
      \eqfinp
    \label{eq:coupling_CAPRA}
  \end{equation}
  \label{de:coupling_CAPRA}
\end{definition}
We stress the point that, in~\eqref{eq:coupling_CAPRA},
the Euclidian scalar product \( \proscal{\primal}{\dual} \)
and the norm term \( \TripleNorm{\primal} \) need not be related, 
that is, the norm~$\TripleNorm{\cdot}$ is not necessarily Euclidian.

The coupling \Capra\ has the property of being 
constant along primal rays, hence the acronym~\Capra\
(Constant Along Primal RAys).
We introduce 
the primal \emph{normalization mapping}~$\normalized$,
from $\RR^d$ towards the unit sphere \( \TripleNormSphere \)
united with $\{0\}$, 
as follows:
\begin{equation}
  \normalized : \RR^d \to \TripleNormSphere \cup \{0\} 
  \eqsepv
  \normalized\np{\primal}=
\frac{\primal}{\TripleNorm{\primal}}
    \mtext{ if } \primal \neq 0 
     \mtext{ and }
 \normalized\np{0}= 0
\eqfinp
  \label{eq:normalization_mapping}
\end{equation}
With these notations, the coupling~\Capra\ 
in~\eqref{eq:coupling_CAPRA} is a special case
of one-sided linear coupling, 
the \emph{Fenchel coupling after primal normalization}:
\(   \CouplingCapra\np{\primal, \dual} = 
\proscal{\normalized\np{\primal}}{\dual} \),
\( \forall \primal \in \RR^d \), \( \forall \dual \in \RR^d \).
We will see below that the \Capra-conjugacy, 
induced by the coupling~\Capra, shares some relations 
with the Fenchel conjugacy 
(see Appendix~\ref{Appendix}).

\subsubsubsection{\Capra-conjugates and biconjugates}

Here are expressions for the \Capra-conjugates and biconjugates
of a function. The following Proposition simply is
\cite[Proposition~2.5]{Chancelier-DeLara:2021_ECAPRA_JCA}
with the normalization mapping~\(\normalized\) in~\eqref{eq:normalization_mapping}. 

\begin{subequations}
  \begin{proposition}
    \label{pr:CAPRA_Fenchel-Moreau_conjugate}
    For any function \( \fonctiondual : \RR^d \to \barRR \), 
    the $\CouplingCapra'$-Fenchel-Moreau conjugate is given by 
    \begin{equation}
      \SFMr{\fonctiondual}{\CouplingCapra}=
      \LFMr{ \fonctiondual }\circ\normalized
      \eqfinp 
      \label{eq:CAPRA'_Fenchel-Moreau_conjugate}
    \end{equation}
    For any function \( \fonctionprimal : \RR^d \to \barRR \), 
    the $\CouplingCapra$-Fenchel-Moreau conjugate is given by 
    \begin{equation}
      \SFM{\fonctionprimal}{\CouplingCapra}=
      \LFM{ \bp{\InfimalPostComposition{\normalized}{\fonctionprimal}} }
      \eqfinv 
      \label{eq:CAPRA_Fenchel-Moreau_conjugate}
    \end{equation}
    where the \conditionalinfimum\
    \( \ConditionalInfimum{\normalized}{\fonctionprimal} \),
 defined in \cite[Definition~2.4]{Chancelier-DeLara:2021_ECAPRA_JCA},
    has the expression
    \begin{equation}
      \ConditionalInfimum{\normalized}{\fonctionprimal}\np{\primal}
      =
      \inf\defset{\fonctionprimal\np{\primal'}}{
        \normalized\np{\primal'}=\primal}
      =
      \begin{cases}
        \inf_{\lambda > 0} \fonctionprimal\np{\lambda\primal}
        & \text{if } \primal \in \TripleNormSphere  \cup \{0\} 
        \eqfinv 
        \\
        +\infty  
        & \text{if } \primal \not\in \TripleNormSphere  \cup \{0\} 
        \eqfinv 
      \end{cases}
      \label{eq:CAPRA_ConditionalInfimum}
    \end{equation}
    and the $\CouplingCapra$-Fenchel-Moreau biconjugate 
    is given by
    \begin{equation}
      \SFMbi{\fonctionprimal}{\CouplingCapra}
      = 
      \LFMr{ \bp{ \SFM{\fonctionprimal}{\CouplingCapra} } }
      \circ \normalized
      =
      \LFMbi{ \bp{\ConditionalInfimum{\normalized}{\fonctionprimal}} }
      \circ \normalized
      \eqfinp
      \label{eq:CAPRA_Fenchel-Moreau_biconjugate}
    \end{equation}
  \end{proposition}
\end{subequations}
The $\CouplingCapra$-Fenchel-Moreau conjugate 
\( \SFM{\fonctionprimal}{\CouplingCapra} \)
is a closed convex function 
(see Appendix~\ref{Appendix}). 

\subsubsubsection{\Capra-convex functions} 

We recall that so-called \emph{$\CouplingCapra$-convex functions} 
are all functions of the form 
$\SFM{ \fonctiondual }{\CouplingCapra'}$, 
for any \( \fonctiondual : \RR^d \to \barRR \),
or, equivalently,
all functions of the form $\SFMbi{\fonctionprimal}{\CouplingCapra}$, 
for any \( \fonctionprimal : \RR^d \to \barRR \),
or, equivalently,
all functions that are equal to their $\CouplingCapra$-biconjugate
(\( \SFMbi{\fonctionprimal}{\CouplingCapra}=\fonctionprimal \))
\cite{Singer:1997,Rubinov:2000,Martinez-Legaz:2005}.
We recall that a function is closed convex on~$\RR^d$ if and only if it 
is either a proper convex lower semi continuous (lsc) function  
or one of the two constant functions~$-\infty$ or $+\infty$
(see Appendix~\ref{Appendix}). 
The following Proposition simply is
\cite[Proposition~2.6]{Chancelier-DeLara:2021_ECAPRA_JCA}
with the normalization mapping~\(\normalized\) in~\eqref{eq:normalization_mapping}.

\begin{proposition}
  \label{pr:CAPRA_convex_functions}  
  A function is $\CouplingCapra$-convex 
  if and only if it is the composition of 
  a closed convex function on~$\RR^d$
  with the normalization mapping~\eqref{eq:normalization_mapping}.
  More precisely, for any function \( \fonctionuncertain : \RR^d \to \barRR \),
  we have the equivalences
  \begin{subequations}
    \begin{align*}
     \fonctionuncertain \textrm{ is }  \CouplingCapra\textrm{-convex }
                                      & \Leftrightarrow
                                         \fonctionuncertain =
                                         \SFMbi{\fonctionuncertain}{\CouplingCapra}
      \\
                                       &\Leftrightarrow 
                                         \fonctionuncertain = \LFMr{ \bp{ \SFM{\fonctionuncertain}{\CouplingCapra} } } 
                                         \circ \, \normalized
                                         \textrm{ (where } \LFMr{ \bp{ \SFM{\fonctionuncertain}{\CouplingCapra} } } 
                                         \textrm{ is a closed convex function) }
      \\
                                       &\Leftrightarrow \textrm{there exists a
                                         closed convex function }
                                         \fonctionprimal: \RR^d  \to \barRR 
                                         \textrm{ such that }
                                         \fonctionuncertain = \fonctionprimal \circ \normalized
                                         \eqfinp
    \end{align*}
  \end{subequations}
\end{proposition}
For instance, letting $\Norm{\cdot}$ be any norm on~$\RR^d$ (not necessarily the Euclidian norm),
the function \( \Norm{\cdot} / \TripleNorm{\cdot} \)
(with the value~0 at~0) is $\CouplingCapra$-convex.

\subsubsubsection{\Capra-subdifferential}

Following the definition of the subdifferential
of a function with respect to a duality 
in \cite{Akian-Gaubert-Kolokoltsov:2002},
the \emph{\Capra-subdifferential} of 
the function \( \fonctionprimal : \RR^d \to \barRR \) 
at~\( \primal \in  \RR^d \) has the following expressions
\begin{subequations}
  \begin{align}
    \partial_{\CouplingCapra}\fonctionprimal\np{\primal} 
    &=
      \defset{ \dual \in \RR^d }{ %
      \SFM{\fonctionprimal}{\CouplingCapra}\np{\dual}=
      \CouplingCapra\np{\primal, \dual} 
      \LowPlus \bp{ -\fonctionprimal\np{\primal} } }
      \label{eq:Capra-subdifferential_b}
    \\
    &=
      \defset{ \dual \in \RR^d }{ %
      \LFM{ \bp{\ConditionalInfimum{\normalized}{\fonctionprimal}} }\np{\dual}=
      \proscal{\normalized\np{\primal}}{\dual} 
      \LowPlus \bp{ -\fonctionprimal\np{\primal} } }
      \eqfinv
      \label{eq:Capra-subdifferential_c}
      \intertext{so that, thanks to the definition~\eqref{eq:normalization_mapping}
      of the normalization mapping~$\normalized$, we deduce that}
      \partial_{\CouplingCapra}\fonctionprimal\np{0} 
    &=
      \defset{ \dual \in \RR^d }{ %
      \LFM{ \bp{\ConditionalInfimum{\normalized}{\fonctionprimal}} }\np{\dual}=
      -\fonctionprimal\np{0} }
      \label{eq:Capra-subdifferential_zero}
    \\
    \partial_{\CouplingCapra}\fonctionprimal\np{\primal} 
    &=
      \defset{ \dual \in \RR^d }{ %
      \LFM{ \bp{\ConditionalInfimum{\normalized}{\fonctionprimal}} }\np{\dual}=
      \frac{ \proscal{\primal}{\dual} }{ \TripleNorm{\primal} } 
      \LowPlus \bp{ -\fonctionprimal\np{\primal} } }
      \eqsepv \forall \primal \in \RR^d\backslash\{0\} 
      \eqfinp 
      \label{eq:Capra-subdifferential_neq_zero}
  \end{align}
  \label{eq:Capra-subdifferential}
\end{subequations}

Now, we turn to analyze the \lzeropseudonorm\ by means of the \Capra\
conjugacy.

\subsection{\Capra-conjugates related to the \lzeropseudonorm}
\label{CAPRA_conjugates_related_to_the_pseudo_norm}

With the Fenchel conjugacy, we calculate that 
\( \LFM{ \delta_{ \LevelSet{\lzero}{k} } }= \delta_{\{0\}} \) 
for all $k\in\ic{1,d} $ --- 
where \( \delta_{ \LevelSet{\lzero}{k} } \) is the characteristic function
of the level sets~\eqref{eq:pseudonormlzero_level_set} --- 
and that 
\( \LFM{ \lzero }= \delta_{\{0\}} \).
Hence, the Fenchel conjugacy is not suitable
to handle the \lzeropseudonorm.

By contrast, we will now show that functions of 
the \lzeropseudonorm\ in~\eqref{eq:pseudo_norm_l0}
--- including the \lzeropseudonorm\ itself and 
the characteristic functions \( \delta_{ \LevelSet{\lzero}{k} } \) 
of its level sets~\eqref{eq:pseudonormlzero_level_set} ---
are related to 
the sequence of dual coordinate-$k$ norms in Definition~\ref{de:coordinate_norm}
by the following \Capra-conjugacy formulas.

\begin{proposition}
  \label{pr:pseudonormlzero_conjugate_varphi}
  Let $\TripleNorm{\cdot}$ be a norm on~$\RR^d$,
  with associated sequence 
  \( \sequence{\CoordinateNormDual{\TripleNorm{\cdot}}{\LocalIndex}}{\LocalIndex\in\ic{1,d}} \)
  of dual coordinate-$k$ norms in Definition~\ref{de:coordinate_norm},
  and 
  associated \Capra-coupling $\CouplingCapra$ in~\eqref{eq:coupling_CAPRA}. 

  For any function \( \varphi : \ic{0,d} \to \barRR \), we have
(with the convention \( \CoordinateNormDual{\TripleNorm{\cdot}}{0} = 0 \))
  \begin{equation}
    \SFM{ \np{ \varphi \circ  \lzero } }{\CouplingCapra} 
    =
    \sup_{\LocalIndex\in\ic{0,d}} \Bc{ \CoordinateNormDual{\TripleNorm{\cdot}}{\LocalIndex} -\varphi\np{\LocalIndex} }  
    \eqfinp
    \label{eq:conjugate_l0norm_varphi}
  \end{equation}
\end{proposition}

\begin{proof} 
  We prove~\eqref{eq:conjugate_l0norm_varphi}:
  \begin{align*}
    \SFM{\np{ \varphi \circ \lzero } }{\CouplingCapra}
    &= 
      \SFM{\Bp{ \inf_{\LocalIndex\in\ic{0,d}} \bc{ \delta_{ \LevelCurve{\lzero}{\LocalIndex} } \UppPlus
      \varphi\np{\LocalIndex} } } }{\CouplingCapra}
       \intertext{because \( \varphi \circ \lzero =
      \inf_{\LocalIndex\in\ic{0,d}} \bc{ \delta_{ \LevelCurve{\lzero}{\LocalIndex} } \UppPlus
      \varphi\np{\LocalIndex} } \)
      since \( \varphi \circ \lzero \) takes the values
      \( \varphi\np{\LocalIndex} \) on the level curves
      \( \LevelCurve{\lzero}{\LocalIndex} \) of~\( \lzero \)
      in~\eqref{eq:pseudonormlzero_level_curve}}
    &= 
      \sup_{\LocalIndex\in\ic{0,d}} 
      \SFM{ \bc{ \delta_{ \LevelCurve{\lzero}{\LocalIndex} } \UppPlus \varphi\np{\LocalIndex} } }{\CouplingCapra}
      \intertext{as conjugacies, being dualities, turn infima into suprema}
    &= 
      \sup_{\LocalIndex\in\ic{0,d}} \bc{ 
      \SFM{ \delta_{ \LevelCurve{\lzero}{\LocalIndex} }}{\CouplingCapra}
      \LowPlus \np{-\varphi\np{\LocalIndex}} }  
      \tag{by property of conjugacies}
    \\
    &= 
      \sup_{\LocalIndex\in\ic{0,d}} \bc{ 
      \sigma_{ \normalized\np{\LevelCurve{\lzero}{\LocalIndex} } } 
      \LowPlus \np{-\varphi\np{\LocalIndex}} }  
      \tag{as \( \SFM{ \delta_{ \LevelCurve{\lzero}{\LocalIndex} }}{\CouplingCapra}
      = \sigma_{ \normalized\np{\LevelCurve{\lzero}{\LocalIndex} } } \)
by \cite[Proposition~2.5]{Chancelier-DeLara:2021_ECAPRA_JCA} }
    \\
    &= 
      \sup_{\LocalIndex\in\ic{0,d}} \Ba{ 
      \sup \ba{ 0, \sigma_{ \LevelCurve{\lzero}{\LocalIndex} \cap \TripleNormSphere } }
      \LowPlus \np{-\varphi\np{\LocalIndex}} }
      \intertext{as \( \normalized\np{\LevelCurve{\lzero}{\LocalIndex} }
      = \{0\} \cup \bp{ \LevelCurve{\lzero}{\LocalIndex} \cap \TripleNormSphere } \) 
      by~\eqref{eq:normalization_mapping}, and as the support function turns a union of sets into a supremum}  
      \phantom{\SFM{\np{ \varphi \circ  \lzero } }{\CouplingCapra}}
    &= 
      \sup_{\LocalIndex\in\ic{0,d}} \ba{ 
      \sigma_{ \LevelCurve{\lzero}{\LocalIndex} \cap \TripleNormSphere } 
      \LowPlus \np{-\varphi\np{\LocalIndex}} }
      \tag{as \( \sigma_{ \LevelCurve{\lzero}{\LocalIndex} \cap
      \TripleNormSphere } \geq 0 \) 
      since \( \LevelCurve{\lzero}{\LocalIndex} \cap \TripleNormSphere =
      -\bp{ \LevelCurve{\lzero}{\LocalIndex} \cap \TripleNormSphere } \) }
    \\
    &= 
      \sup \Ba{ -\varphi\np{0}, 
      \sup_{\LocalIndex\in\ic{1,d}} 
      \Bc{ \CoordinateNormDual{\TripleNorm{\dual}}{\LocalIndex} -\varphi\np{\LocalIndex} } }
      \tag{as \( \sigma_{ \LevelCurve{\lzero}{\LocalIndex} \cap \TripleNormSphere } =
      \CoordinateNormDual{\TripleNorm{\cdot}}{\LocalIndex} \)
      by~\eqref{eq:dual_coordinate_norm} } 
    \\
    &= 
      \sup_{\LocalIndex\in\ic{0,d}} 
      \Bc{ \CoordinateNormDual{\TripleNorm{\dual}}{\LocalIndex} -\varphi\np{\LocalIndex} } 
      \tag{using the convention that \( \CoordinateNormDual{\TripleNorm{\cdot}}{0} = 0 \) }
      \eqfinp
  \end{align*}

  This ends the proof.
\end{proof}

With \( \varphi \) the identity function on~\( \ic{0,d} \),
we find the \Capra-conjugate of the \lzeropseudonorm.
With the functions
\( \varphi=\delta_{\ic{0,k}} \) (for any $k \in \ic{0,d} $),
we find the \Capra-conjugates of the
characteristic functions \( \delta_{ \LevelSet{\lzero}{k} } \)
of its level sets~\eqref{eq:pseudonormlzero_level_set}.
The corresponding expressions are given in Table~\ref{tab:results_conjugacy}.

\subsection{\Capra-biconjugates related to the \lzeropseudonorm}
\label{CAPRA_biconjugates_related_to_the_pseudo_norm}

With the Fenchel conjugacy, we calculate that 
\( \LFMbi{ \delta_{ \LevelSet{\lzero}{k} } }= 0 \),
for all $k\in\ic{1,d} $, and that 
\( \LFMbi{ \lzero }= 0 \).
Hence, the Fenchel conjugacy is not suitable
to handle the \lzeropseudonorm.

By contrast, we will now show that functions of 
the \lzeropseudonorm\ in~\eqref{eq:pseudo_norm_l0}
--- including the \lzeropseudonorm\ itself and 
the characteristic functions \( \delta_{ \LevelSet{\lzero}{k} } \) 
of its level sets~\eqref{eq:pseudonormlzero_level_set} ---
are related to 
the sequences of coordinate-$k$ norms and dual 
coordinate-$k$ norms in Definition~\ref{de:coordinate_norm}
by the following \Capra-biconjugacy formulas.

\begin{proposition}
  Let $\TripleNorm{\cdot}$ be a norm on~$\RR^d$,
  with associated sequence
  \( \sequence{\CoordinateNorm{\TripleNorm{\cdot}}{\LocalIndex}}{\LocalIndex\in\ic{1,d}} \)
  of coordinate-$k$ norms and sequence
  \( \sequence{\CoordinateNorm{\TripleNormDual{\cdot}}{\LocalIndex}}{\LocalIndex\in\ic{1,d}} \)
  of dual coordinate-$k$ norms, as 
  in Definition~\ref{de:coordinate_norm},
  and with
  associated \Capra\ coupling $\CouplingCapra$ in~\eqref{eq:coupling_CAPRA}. 

  \begin{enumerate}
  \item 
    For any function \( \varphi : \ic{0,d} \to \barRR \), we have
    \begin{subequations}
      \begin{align}
        \SFMbi{ \np{ \varphi \circ \lzero} }{\CouplingCapra}\np{\primal}
        &=
          \LFMr{ \bp{ \SFM{\np{ \varphi \circ \lzero}}{\CouplingCapra} } }
          \np{ \frac{\primal}{\TripleNorm{\primal}} }
          \eqsepv \forall \primal \in \RR^d\backslash\{0\}
          \label{eq:Biconjugate_ofvarphi_of_lzero}
          \eqfinv 
          \intertext{where the closed convex function 
          \( \LFMr{ \bp{ \SFM{\np{ \varphi \circ \lzero}}{\CouplingCapra} } } \) 
          has the following expression as a Fenchel conjugate}
          \LFMr{ \bp{ \SFM{\np{ \varphi \circ \lzero}}{\CouplingCapra} } }
        &=
          \LFMr{ \Bp{ 
          \sup_{\LocalIndex\in\ic{0,d}} \bc{
          \CoordinateNormDual{\TripleNorm{\cdot}}{\LocalIndex} -\varphi\np{\LocalIndex} } } }
          \eqfinv 
          \intertext{and also
          has the following four expressions as a Fenchel biconjugate}
        &=
          \LFMbi{ \Bp{ \inf_{\LocalIndex\in\ic{0,d}} \bc{ 
          \delta_{  \CoordinateNorm{\TripleNormBall}{\LocalIndex} } \UppPlus \varphi\np{\LocalIndex} } } }
          \eqfinv 
          \label{eq:Biconjugate_of_min_balls_ind}
          \intertext{hence the function~\( \LFMr{ \bp{ \SFM{\np{ \varphi \circ \lzero}}{\CouplingCapra} } } \)
          is the largest closed convex function
          below the integer valued function 
          \( \inf_{\LocalIndex\in\ic{0,d}} \bc{ 
          \delta_{  \CoordinateNorm{\TripleNormBall}{\LocalIndex} } \UppPlus \varphi\np{\LocalIndex} } \),
          such that 
 \( \primal \in \CoordinateNorm{\TripleNormBall}{\LocalIndex} 
          \backslash \CoordinateNorm{\TripleNormBall}{\LocalIndex-1} \mapsto \varphi\np{\LocalIndex} \) 
          for $l\in\ic{1,d}$,
          and $\primal \in \CoordinateNorm{\TripleNormBall}{0} = \{0\} \mapsto \varphi\np{0}$, the function
          being infinite outside~\( \CoordinateNorm{\TripleNormBall}{d}=
          \TripleNormBall \), that is, with the convention that \( \CoordinateNorm{\TripleNormBall}{0}=\{0\} \)
          and that \( \inf \emptyset = +\infty \)}
        &= 
          \LFMbi{ \Bp{ \primal \mapsto \inf \bset{ \varphi\np{\LocalIndex} }%
          { \primal \in \CoordinateNorm{\TripleNormBall}{\LocalIndex}
          \eqsepv \LocalIndex \in \ic{0,d} } } }
          \eqfinv 
          \label{eq:Biconjugate_of_min_balls_ind_bis}
          \\
        &=
          \LFMbi{ \Bp{ \inf_{\LocalIndex\in\ic{0,d}} \bc{ 
          \delta_{ \CoordinateNorm{\TripleNormSphere}{\LocalIndex} } \UppPlus \varphi\np{\LocalIndex} } } }
          \eqfinv 
          \label{eq:Biconjugate_of_min_spheres_ind}
          \intertext{hence the function~\( \LFMr{ \bp{ \SFM{\np{ \varphi \circ \lzero}}{\CouplingCapra} } } \)
          is the largest closed convex function
          below the integer valued function 
          \( \inf_{\LocalIndex\in\ic{0,d}} \bc{ 
          \delta_{  \CoordinateNorm{\TripleNormSphere}{\LocalIndex} } \UppPlus
          \varphi\np{\LocalIndex} } \), that is,
with the convention that \( \CoordinateNorm{\TripleNormSphere}{0}=\{0\} \)
          and that \( \inf \emptyset = +\infty \)}
        &= 
          \LFMbi{ \Bp{\primal \mapsto \inf \bset{ \varphi\np{\LocalIndex} }%
          { \primal \in \CoordinateNorm{\TripleNormSphere}{\LocalIndex}
          \eqsepv \LocalIndex \in \ic{0,d} } } }
          \eqfinp 
          \label{eq:Biconjugate_of_min_spheres_ind_bis}
       \end{align}
    \item 
      For any function \( \varphi : \ic{0,d} \to \RR \), 
      that is, with finite values, the function 
      \( \LFMr{ \bp{ \SFM{\np{ \varphi \circ \lzero}}{\CouplingCapra} } } \) 
      is proper convex lsc and 
      has the following variational expression
(where \( \Delta_{d+1} \) denotes the simplex of~$\RR^{d+1}$)
      \begin{align}
        \LFMr{ \bp{ \SFM{\np{ \varphi \circ \lzero}}{\CouplingCapra} } }\np{\primal}
        &= 
          \min_{ \substack{%
          \np{\lambda_0,\lambda_1,\ldots,\lambda_d} \in \Delta_{d+1} 
        \\
        \primal \in \sum_{ \LocalIndex=1 }^{ d } \lambda_{\LocalIndex} \CoordinateNorm{\TripleNormBall}{\LocalIndex} 
        } } 
        \sum_{ \LocalIndex=0}^{ d } \lambda_{\LocalIndex}
        \varphi\np{\LocalIndex} 
        \eqsepv \forall  \primal \in \RR^d 
        \eqfinp
        \label{eq:biconjugate_with_balls}
      \end{align}
    \item
      \label{it:biconjugate_l0norm_varphi}
      For any function \( \varphi : \ic{0,d} \to \RR_+ \), 
      that is, with nonnegative finite values,
      and such that \( \varphi\np{0}=0 \), the function 
      \( \LFMr{ \bp{ \SFM{\np{ \varphi \circ \lzero}}{\CouplingCapra} } } \) 
      is proper convex lsc and 
      has the following two variational expressions\footnote{%
In~\eqref{eq:biconjugate_with_balls}, the sum starts from \( \LocalIndex=0 \),
        whereas in~\eqref{eq:biconjugate_with_spheres}
        and in~\eqref{eq:pseudonormlzero_convex_minimum}, 
        the sum starts from \( \LocalIndex=1 \)}
      \begin{align}
        \LFMr{ \bp{ \SFM{\np{ \varphi \circ \lzero}}{\CouplingCapra} } }\np{\primal}
        &= 
          \min_{ \substack{%
          \np{\lambda_0,\lambda_1,\ldots,\lambda_d} \in \Delta_{d+1} 
        \\
        \primal \in \sum_{ \LocalIndex=1 }^{ d } \lambda_{\LocalIndex} \CoordinateNorm{\TripleNormSphere}{\LocalIndex} 
        } } 
        \sum_{ \LocalIndex=1 }^{ d } \lambda_{\LocalIndex} \varphi\np{\LocalIndex} 
        \eqsepv \forall \primal \in \RR^d
        \eqfinv
        \label{eq:biconjugate_with_spheres}
        \\
        &= 
          \min_{ \substack{%
          z^{(1)} \in \RR^d, \ldots, z^{(d)} \in \RR^d 
        \\
        \sum_{ \LocalIndex=1 }^{ d } \CoordinateNorm{\TripleNorm{z^{(\LocalIndex)}}}{\LocalIndex} \leq 1
        \\
        \sum_{ \LocalIndex=1 }^{ d } z^{(\LocalIndex)} = \primal } }
        \sum_{ \LocalIndex=1 }^{ d } \varphi\np{\LocalIndex}
        \CoordinateNorm{\TripleNorm{z^{(\LocalIndex)}}}{\LocalIndex} 
        \eqsepv \forall \primal \in \RR^d
        \eqfinv
        \label{eq:pseudonormlzero_convex_minimum}
      \end{align}
      %
    \end{subequations}
    and the function \( \SFMbi{ \np{ \varphi \circ \lzero} }{\CouplingCapra} \)
    has the following variational expression
    \begin{equation}
      \SFMbi{ \np{ \varphi \circ \lzero} }{\CouplingCapra}\np{\primal}
      =
      \frac{ 1 }{ \TripleNorm{\primal} } 
      \min_{ \substack{%
          z^{(1)} \in \RR^d, \ldots, z^{(d)} \in \RR^d 
          \\
          \sum_{ \LocalIndex=1 }^{ d } \CoordinateNorm{\TripleNorm{z^{(\LocalIndex)}}}{\LocalIndex} \leq \TripleNorm{\primal}
          \\
          \sum_{ \LocalIndex=1 }^{ d } z^{(\LocalIndex)} = \primal } }
      \sum_{ \LocalIndex=1 }^{ d }
      \CoordinateNorm{\TripleNorm{z^{(\LocalIndex)}}}{\LocalIndex} \varphi\np{\LocalIndex} 
      \eqsepv \forall \primal \in \RR^d\backslash\{0\} 
      \eqfinp
      \label{eq:biconjugate_l0norm_varphi}
    \end{equation}
  \end{enumerate}
  \label{pr:pseudonormlzero_biconjugate_varphi}
\end{proposition}

\begin{proof} 
  We first note that \( \SFMbi{\np{ \varphi \circ \lzero}}{\CouplingCapra}
  =\LFMr{ \bp{ \SFM{\np{ \varphi \circ \lzero}}{\CouplingCapra} } }
  \circ \normalized \), by~\eqref{eq:CAPRA_Fenchel-Moreau_biconjugate},
  and we study \( \LFMr{ \bp{ \SFM{\np{ \varphi \circ \lzero}}{\CouplingCapra} } } \).
  \medskip

  \begin{enumerate}
  \item 
    Let \( \varphi : \ic{0,d} \to \barRR \) be a function.
    The equality~\eqref{eq:Biconjugate_ofvarphi_of_lzero} is a straightforward
    consequence of the expression~\eqref{eq:CAPRA_Fenchel-Moreau_biconjugate}
    for a \Capra-biconjugate,
    and of the fact that \( \normalized\np{\primal}= \frac{\primal}{\TripleNorm{\primal}} \)
    when \( \primal \neq 0 \) by~\eqref{eq:normalization_mapping}.
    We have
    \begin{align*}
           \LFMr{ \bp{ \SFM{\np{ \varphi \circ \lzero}}{\CouplingCapra} } }
         &= 
          \LFMr{ \Bp{ 
          \sup_{\LocalIndex\in\ic{0,d}} \bc{ \CoordinateNormDual{\TripleNorm{\cdot}}{\LocalIndex} -\varphi\np{\LocalIndex} } } }
          \tag{by~\eqref{eq:conjugate_l0norm_varphi} }
      \\
        &=
          \LFMr{ \Bp{ \sup_{\LocalIndex\in\ic{0,d}} \bc{ \sigma_{ \CoordinateNorm{\TripleNormBall}{\LocalIndex} } -\varphi\np{\LocalIndex} } } }
          \intertext{by~\eqref{eq:norm_dual_norm} as \( \CoordinateNorm{\TripleNormBall}{\LocalIndex}
          \) is the unit ball of the norm~\( \CoordinateNorm{\TripleNorm{\cdot}}{\LocalIndex} \)
          by~\eqref{eq:coordinate_norm_unit_sphere_ball}
          and with the convention 
          \( \CoordinateNorm{\TripleNormBall}{0} =\{0\} \)}
        &=
          \LFMr{ \Bp{ \sup_{\LocalIndex\in\ic{0,d}} \bc{ 
          \LFM{ \delta_{  \CoordinateNorm{\TripleNormBall}{\LocalIndex} } } - \varphi\np{\LocalIndex} } } }
          \tag{because 
          \( \LFM{ \delta_{  \CoordinateNorm{\TripleNormBall}{\LocalIndex} } }=
          \sigma_{ \CoordinateNorm{\TripleNormBall}{\LocalIndex} } \) }
      \\
        &=
          \LFMr{ \Bp{ \sup_{\LocalIndex\in\ic{0,d}} 
          \LFM{ \bp{ \delta_{  \CoordinateNorm{\TripleNormBall}{\LocalIndex} } + \varphi\np{\LocalIndex} } } } }
          \tag{by property of conjugacies}
      \\
        &=
          \LFMr{ \bgp{ \LFM{ \Bp{ \inf_{\LocalIndex\in\ic{0,d}} \bc{ 
          \delta_{  \CoordinateNorm{\TripleNormBall}{\LocalIndex} } + \varphi\np{\LocalIndex} } } } } }
          \intertext{as conjugacies, being dualities, turn infima into suprema}
        &=
          \LFMbi{ \Bp{ \inf_{\LocalIndex\in\ic{0,d}} \bc{ 
          \delta_{  \CoordinateNorm{\TripleNormBall}{\LocalIndex} } + \varphi\np{\LocalIndex} } } }
          \tag{by~\eqref{eq:Fenchel_biconjugate}}
          \eqfinp
    \end{align*}
    Thus, we have obtained~\eqref{eq:Biconjugate_of_min_balls_ind}
    and \eqref{eq:Biconjugate_of_min_balls_ind_bis}.
    Now, if we follow again the above sequence of equalities, 
    we see that, everywhere, we can replace 
    the balls~\( \CoordinateNorm{\TripleNormBall}{\LocalIndex} \) by
    the spheres~\( \CoordinateNorm{\TripleNormSphere}{\LocalIndex} \), since
    \( \CoordinateNormDual{\TripleNorm{\cdot}}{\LocalIndex} =
    \sigma_{ \CoordinateNorm{\TripleNormSphere}{\LocalIndex} } =
    \LFM{ \delta_{  \CoordinateNorm{\TripleNormSphere}{\LocalIndex} } } \).
    Thus, we obtain~\eqref{eq:Biconjugate_of_min_spheres_ind}
    and \eqref{eq:Biconjugate_of_min_spheres_ind_bis}.
 
  \item 
    Let \( \varphi : \ic{0,d} \to \RR \) be a function.
    Then the closed convex function \( \LFMr{ \bp{ \SFM{\np{ \varphi \circ \lzero}}{\CouplingCapra} } } \) 
    is proper. Indeed, on the one hand, it is easily seen that the function
    \( \SFM{\np{ \varphi \circ \lzero}}{\CouplingCapra} \) takes finite values,
    from which we deduce that the function
    \( \LFMr{ \bp{ \SFM{\np{ \varphi \circ \lzero}}{\CouplingCapra} } } \) 
    never takes the value~$-\infty$.
    On the other hand, by~\eqref{eq:Biconjugate_ofvarphi_of_lzero} 
    and by the inequality \( \SFMbi{\np{ \varphi \circ \lzero } }{\CouplingCapra}
    \leq \varphi \circ \lzero \) obtained from~\eqref{eq:galois-cor},
    we deduce that the function
    \( \LFMr{ \bp{ \SFM{\np{ \varphi \circ \lzero}}{\CouplingCapra} } } \) 
    never takes the value~$+\infty$ on the unit sphere.
    Therefore, the\( \LFMr{ \bp{ \SFM{\np{ \varphi \circ \lzero}}{\CouplingCapra} } } \) 
    is proper.
    \medskip

    For the remaining expressions for
    \( \LFMr{ \bp{ \SFM{\np{ \varphi \circ \lzero}}{\CouplingCapra} } } \),
    we use a formula 
    \cite[Corollary~2.8.11]{Zalinescu:2002}
    for the Fenchel conjugate of the supremum of proper convex functions
    \(\fonctionprimal_{\LocalIndex} : \RR^d \to \barRR \),
    $\LocalIndex\in\ic{0,n}$: 
    \begin{equation}
      \label{eq:Fenchel_conjugate_of_the_supremum_of_proper_convex_functions}
      \bigcap_{\LocalIndex=0,1,\ldots,n} \dom\fonctionprimal_{\LocalIndex} \neq \emptyset
      \implies
      \LFM{ \bp{ \sup_{\LocalIndex=0,1,\ldots,n} \fonctionprimal_{\LocalIndex} } }
      =
      \min_{\np{\lambda_0,\lambda_1,\ldots,\lambda_n}\in \Delta_{n+1}} 
      \LFM{ \Bp{ \sum_{ \LocalIndex=0}^{n} \lambda_{\LocalIndex} \fonctionprimal_{\LocalIndex} } }
      \eqfinv
    \end{equation}
    where \( \dom\fonctionprimal= 
    \defset{\primal\in\RR^d}{\fonctionprimal\np{\primal}<+\infty} \) is 
    the effective domain 
    (see Appendix~\ref{Appendix}), 
    and where \( \Delta_{n+1} \) is the simplex of~$\RR^{n+1}$.
    We obtain 
    \begin{align*}
          \LFMr{ \bp{ \SFM{\np{ \varphi \circ \lzero}}{\CouplingCapra} } }
        &= 
          \LFMr{ \bp{ 
          \sup_{\LocalIndex\in\ic{0,d}} \Bc{ \CoordinateNormDual{\TripleNorm{\cdot}}{\LocalIndex} -\varphi\np{\LocalIndex} } } }
          \tag{by~\eqref{eq:conjugate_l0norm_varphi} }
      \\
        &=
          \LFMr{ \Bp{ \sup_{\LocalIndex\in\ic{0,d}} \Bc{ \sigma_{ \CoordinateNorm{\TripleNormBall}{\LocalIndex} } -\varphi\np{\LocalIndex} } } }
          \intertext{by~\eqref{eq:norm_dual_norm} as \( \CoordinateNorm{\TripleNormBall}{\LocalIndex}
          \) is the unit ball of the norm~\( \CoordinateNorm{\TripleNorm{\cdot}}{\LocalIndex} \)
           by~\eqref{eq:coordinate_norm_unit_sphere_ball}
          and with 
          \( \CoordinateNorm{\TripleNormBall}{0} =\{0\} \)}
        &=
          \min_{\np{\lambda_0,\lambda_1,\ldots,\lambda_d}\in \Delta_{d+1}} 
          \LFMr{ \Bp{ \sum_{ \LocalIndex=0}^{d} \lambda_{\LocalIndex} 
          \Bc{ \sigma_{ \CoordinateNorm{\TripleNormBall}{\LocalIndex} } -\varphi\np{\LocalIndex} 
          } } } 
          \tag{
          by~\eqref{eq:Fenchel_conjugate_of_the_supremum_of_proper_convex_functions} }
          \intertext{by \cite[Corollary~2.8.11]{Zalinescu:2002}, as the functions
          \(\fonctionprimal_{\LocalIndex} =  \sigma_{ \CoordinateNorm{\TripleNormBall}{\LocalIndex} } -\varphi\np{\LocalIndex} \)
          are proper convex (they even take finite values), for $\LocalIndex\in\ic{0,d}$}
        &=
          \min_{\np{\lambda_0,\lambda_1,\ldots,\lambda_d}\in \Delta_{d+1}} 
          \LFMr{ \Bp{ \sigma_{ \sum_{ \LocalIndex=0}^{d} \lambda_{\LocalIndex}
          \CoordinateNorm{\TripleNormBall}{\LocalIndex} }
          - \sum_{ \LocalIndex=0}^{d} \lambda_{\LocalIndex} \varphi\np{\LocalIndex} 
          } }
          \intertext{as, for all $\LocalIndex\in\ic{1,d}$, 
          \( \lambda_{\LocalIndex} \sigma_{ \CoordinateNorm{\TripleNormBall}{\LocalIndex} } = 
          \sigma_{ \lambda_{\LocalIndex} \CoordinateNorm{\TripleNormBall}{\LocalIndex} } \) since \( \lambda_{\LocalIndex} \geq 0 \), 
          and then using the well-known property that the support function of 
          a Minkowski sum of subsets is the sum of the support functions of 
          the individual subsets \cite[p.~226]{Hiriart-Urruty-Lemarechal-I:1993}}
          &=
          \min_{\np{\lambda_0,\lambda_1,\ldots,\lambda_d}\in \Delta_{d+1}} 
          \LFMr{ \Bp{ \sigma_{ \sum_{ \LocalIndex=1}^{d} \lambda_{\LocalIndex}
          \CoordinateNorm{\TripleNormBall}{\LocalIndex} }
          - \sum_{ \LocalIndex=0}^{d} \lambda_{\LocalIndex} \varphi\np{\LocalIndex} 
          } }
          \tag{thanks to the convention 
          \( \CoordinateNorm{\TripleNormBall}{0} =\{0\} \)}
      \\
        &=
          \min_{\np{\lambda_0,\lambda_1,\ldots,\lambda_d}\in \Delta_{d+1}} 
          \Bp{ \LFMr{ \bp{ \sigma_{ \sum_{\LocalIndex=1}^{d} \lambda_{\LocalIndex}
          \CoordinateNorm{\TripleNormBall}{\LocalIndex} } } }
          + \sum_{\LocalIndex=0}^{d} \lambda_{\LocalIndex} \varphi\np{\LocalIndex} 
          } 
          \tag{by property of conjugacies} 
      \\
        &=
          \min_{\np{\lambda_0,\lambda_1,\ldots,\lambda_d}\in \Delta_{d+1}} \Bp{ 
          \delta_{ \sum_{\LocalIndex=1}^{d} \lambda_{\LocalIndex} \CoordinateNorm{\TripleNormBall}{\LocalIndex} } 
          +
          \sum_{\LocalIndex=0}^{d} \lambda_{\LocalIndex}  \varphi\np{\LocalIndex} }
          \tag{because \( \sum_{\LocalIndex=1}^{d} \lambda_{\LocalIndex}
          \CoordinateNorm{\TripleNormBall}{\LocalIndex} \)
          is a closed convex set. }
          \intertext{Therefore, we deduce that, for all \( \primal \in \RR^d \), }
          \LFMr{ \bp{ \SFM{\np{ \varphi \circ \lzero}}{\CouplingCapra} } }\np{\primal} 
        &= 
          \min_{ \substack{%
          \np{\lambda_0,\lambda_1,\ldots,\lambda_d}\in \Delta_{d+1} 
      \\
      \primal \in \sum_{\LocalIndex=1}^{ d } \lambda_{\LocalIndex} \CoordinateNorm{\TripleNormBall}{\LocalIndex} 
      } } 
      \sum_{\LocalIndex=0}^{ d } \lambda_{\LocalIndex} \varphi\np{\LocalIndex} 
\eqsepv \text{ which is~\eqref{eq:biconjugate_with_balls}.}
    \end{align*}

  \item 
    Let \( \varphi : \ic{0,d} \to \RR_+ \) be a function 
    such that \( \varphi\np{0}=0 \). 
    Then the closed convex function \( \LFMr{ \bp{ \SFM{\np{ \varphi \circ \lzero}}{\CouplingCapra} } } \) 
    is proper, as seen above. 
    We go on with 
    \begin{align*}
      \LFMr{ \bp{ \SFM{\np{ \varphi \circ \lzero}}{\CouplingCapra} } }\np{\primal} 
      &= 
        \min_{ \substack{%
        \np{\lambda_0,\lambda_1,\ldots,\lambda_d}\in \Delta_{d+1} 
      \\
      \primal \in \sum_{\LocalIndex=1}^{ d } \lambda_{\LocalIndex} \CoordinateNorm{\TripleNormBall}{\LocalIndex} 
      } } 
      \sum_{\LocalIndex=1}^{ d } \lambda_{\LocalIndex} \varphi\np{\LocalIndex}
      \tag{because \( \varphi\np{0}=0 \) }
      \\
      &= 
        \min_{ \substack{%
        z^{(1)} \in \CoordinateNorm{\TripleNormBall}{1}, 
        \ldots, z^{(d)} \in \CoordinateNorm{\TripleNormBall}{d} 
      \\
      \lambda_1 \geq 0, \ldots, \lambda_d \geq 0
      \\
      \sum_{ \LocalIndex=1 }^{ d } \lambda_{\LocalIndex} \leq 1 
      \\
      \sum_{ \LocalIndex=1 }^{ d } \lambda_{\LocalIndex} z^{({\LocalIndex})} = \primal
      } }  
      \sum_{ \LocalIndex=1 }^{ d } \lambda_{\LocalIndex} \varphi\np{\LocalIndex} 
      \intertext{because \( \np{\lambda_0,\lambda_1,\ldots,\lambda_d}\in \Delta_{d+1}\) 
      if and only if 
      \(  \lambda_1 \geq 0, \ldots, \lambda_d \geq 0 \) and
      \( \sum_{ \LocalIndex=1 }^{ d } \lambda_{\LocalIndex} \leq 1 \) and 
      \( \lambda_0= 1-\sum_{ \LocalIndex=1 }^{ d } \lambda_{\LocalIndex} \)}
      &=
        \min_{ \substack{%
        s^{(1)} \in \CoordinateNorm{\TripleNormSphere}{1}, 
        \ldots, s^{(d)} \in \CoordinateNorm{\TripleNormSphere}{d} 
      \\
      \mu_1 \geq 0, \ldots, \mu_d \geq 0
      \\
      \sum_{ \LocalIndex=1 }^{ d } \mu_{\LocalIndex} \leq 1 
      \\
      \sum_{ \LocalIndex=1 }^{ d } \mu_{\LocalIndex} s^{({\LocalIndex})} = \primal
      } }  
      \sum_{ \LocalIndex=1 }^{ d } \mu_{\LocalIndex} \varphi\np{\LocalIndex} 
      \end{align*}
because, on the one hand, the inequality $\leq$ is obvious as 
      the unit sphere \( \CoordinateNorm{\TripleNormSphere}{\LocalIndex} \)
      in~\eqref{eq:dual_coordinate_norm_unit_sphere_ball}
      is included in the unit ball \( \CoordinateNorm{\TripleNormBall}{\LocalIndex} \)
      for all $\LocalIndex\in\ic{1,d}$;
      and, on the other hand, 
      the inequality $\geq$ comes from putting,
      for $\LocalIndex\in\ic{1,d}$, 
      \(  \mu_{\LocalIndex} = \lambda_{\LocalIndex} \CoordinateNorm{\TripleNorm{z^{({\LocalIndex})}}}{\LocalIndex} \)
      and observing that 
      i) $\sum_{i=1}^d  \mu_{\LocalIndex} = \sum_{i=1}^d  \lambda_{\LocalIndex} 
      \CoordinateNorm{\TripleNorm{z^{({\LocalIndex})}}}{\LocalIndex} \leq \sum_{i=1}^d  \lambda_{\LocalIndex}\le 1$
      because \( \CoordinateNorm{\TripleNorm{z^{({\LocalIndex})}}}{\LocalIndex}
      \leq 1 \)
      as \( z^{(\LocalIndex)} \in \CoordinateNorm{\TripleNormBall}{\LocalIndex} \) 
      ii) 
      for all $\LocalIndex\in\ic{1,d}$, there exists \( s^{({\LocalIndex})} 
      \in \CoordinateNorm{\TripleNormSphere}{\LocalIndex} \) such that 
      \( \lambda_{\LocalIndex} z^{({\LocalIndex})} = \mu_{\LocalIndex} s^{({\LocalIndex})} \) 
      (take any $s^{({\LocalIndex})}$ when $z^{({\LocalIndex})}=0$
      because $\mu_{\LocalIndex}=0$, and take
      \( s^{({\LocalIndex})}=\frac{z^{({\LocalIndex})}}{ \CoordinateNorm{\TripleNorm{z^{({\LocalIndex})}}}{\LocalIndex} } \) 
      when $z^{({\LocalIndex})} \neq 0$)
      iii) 
      \( \sum_{ \LocalIndex=1 }^{ d } \lambda_{\LocalIndex} \varphi\np{\LocalIndex} \geq
      \sum_{ \LocalIndex=1 }^{ d } \lambda_{\LocalIndex}
      \CoordinateNorm{\TripleNorm{z^{({\LocalIndex})}}}{\LocalIndex} \varphi\np{\LocalIndex} 
      = \sum_{ \LocalIndex=1 }^{ d } \mu_{\LocalIndex} \varphi\np{\LocalIndex} \) 
      because \( 1 \geq \CoordinateNorm{\TripleNorm{z^{({\LocalIndex})}}}{\LocalIndex} \)
      and \( \varphi\np{\LocalIndex} \geq 0 \) 
      \begin{align*}
     \phantom{ \LFMr{ \bp{ \SFM{\np{ \varphi \circ \lzero}}{\CouplingCapra} } }\np{\primal} }
      &=
        \min_{ \substack{%
        z^{(1)} \in \RR^d, \ldots, z^{(d)} \in \RR^d 
      \\
      \sum_{ \LocalIndex=1 }^{ d } \CoordinateNorm{\TripleNorm{z^{({\LocalIndex})}}}{\LocalIndex} \leq 1 
      \\
      \sum_{ \LocalIndex=1 }^{ d } z^{({\LocalIndex})} = \primal
      } }
      \sum_{ \LocalIndex=1 }^{ d } \varphi\np{\LocalIndex}
      \CoordinateNorm{\TripleNorm{z^{({\LocalIndex})}}}{\LocalIndex} 
      \eqfinv
    \end{align*}
    by putting \( z^{({\LocalIndex})} = \mu_{\LocalIndex} s^{({\LocalIndex})}\),
    for all $\LocalIndex\in\ic{1,d}$.
    Thus, we have obtained~\eqref{eq:biconjugate_with_spheres}.

    Finally, from \( \SFMbi{\np{ \varphi \circ \lzero}}{\CouplingCapra}
    =\LFMr{ \bp{ \SFM{\np{ \varphi \circ \lzero}}{\CouplingCapra} } }
    \circ \normalized \), by~\eqref{eq:CAPRA_Fenchel-Moreau_biconjugate},
    we get that 
    \begin{equation*}
      \SFMbi{\np{ \varphi \circ \lzero}}{\CouplingCapra}\np{\primal} 
      = \frac{ 1 }{ \TripleNorm{\primal} } 
      \min_{ \substack{%
          z^{(1)} \in \RR^d, \ldots, z^{(d)} \in \RR^d 
          \\
          \sum_{ \LocalIndex=1 }^{ d } \CoordinateNorm{\TripleNorm{z^{(\LocalIndex)}}}{\LocalIndex} \leq \TripleNorm{\primal}
          \\
          \sum_{ \LocalIndex=1 }^{ d } z^{(\LocalIndex)} = \primal  } }
      \sum_{ \LocalIndex=1 }^{ d } \varphi\np{\LocalIndex} 
      \CoordinateNorm{\TripleNorm{z^{(\LocalIndex)}}}{\LocalIndex} 
      \eqsepv \forall \primal \in \RR^d\backslash\{0\}
      \eqfinv
    \end{equation*}
    where we have used that 
    \( \normalized\np{\primal}= \frac{\primal}{\TripleNorm{\primal}} \)
    when \( \primal \neq 0 \) by~\eqref{eq:normalization_mapping}.
    Therefore, we have proved~\eqref{eq:biconjugate_l0norm_varphi}.

  \end{enumerate}

  \medskip

  This ends the proof.
\end{proof}

Before finishing that part on \Capra-biconjugates, we provide 
the following characterization of when the characteristic functions \( \delta_{ \LevelSet{\lzero}{k} } \)
are $\CouplingCapra$-convex.

\begin{corollary}
  Let $\TripleNorm{\cdot}$ be a norm on~$\RR^d$,
  with associated sequence 
  \( \sequence{\CoordinateNorm{\TripleNorm{\cdot}}{\LocalIndex}}{\LocalIndex\in\ic{1,d}} \)
  of coordinate-$k$ norms in Definition~\ref{de:coordinate_norm}
  and 
  associated \Capra\ coupling $\CouplingCapra$ in~\eqref{eq:coupling_CAPRA}.

  The following statements are equivalent.
  \begin{enumerate}
  \item 
    The sequence 
    \( \sequence{\CoordinateNorm{\TripleNorm{\cdot}}{\LocalIndex}}{\LocalIndex\in\ic{1,d}} \)
    of coordinate-$k$ norms 
    is strictly decreasingly graded with respect to the \lzeropseudonorm,
    as in Definition~\ref{de:decreasingly_graded}.
  \item 
    For all \( k \in \ic{0,d} \), 
    the characteristic functions \( \delta_{ \LevelSet{\lzero}{k} } \)
    are $\CouplingCapra$-convex, that is, 
    \begin{equation}
      \SFMbi{ \delta_{ \LevelSet{\lzero}{k} } }{\CouplingCapra} 
      =
      \delta_{ \LevelSet{\lzero}{k} } 
      \eqsepv k\in\ic{0,d} 
      \eqfinp
    \end{equation}
  \end{enumerate}
\end{corollary}

\begin{proof}
We start by providing an expression for 
\( \SFMbi{ \delta_{ \LevelSet{\lzero}{k} } }{\CouplingCapra} \).
For any \( k \in \ic{0,d} \), we have 
  \begin{align*}
    \SFMbi{ \delta_{ \LevelSet{\lzero}{k} } }{\CouplingCapra}
    &=
      \LFMbi{ \Bp{ \inf_{\LocalIndex\in\ic{0,d}} \bc{ 
      \delta_{  \CoordinateNorm{\TripleNormBall}{\LocalIndex} } \UppPlus
      \delta_{\ic{0,k}}\np{\LocalIndex} } } } \circ \normalized
      \tag{by~\eqref{eq:Biconjugate_of_min_balls_ind}
      with the functions
      \( \varphi=\delta_{\ic{0,k}} \)}
    \\        
    &=
      \LFMbi{ \Bp{ \inf_{\LocalIndex=0,1,\ldots,k} 
      \delta_{ \CoordinateNorm{\TripleNormBall}{\LocalIndex} } } } \circ \normalized
    \\        
    &=
      \LFMbi{ \bp{ \delta_{ \CoordinateNorm{\TripleNormBall}{k} } } } \circ
      \normalized
      \intertext{by the inclusions \( \CoordinateNorm{\TripleNormBall}{1} 
      \subset \cdots \subset
      \CoordinateNorm{\TripleNormBall}{k} \)
      in~\eqref{eq:coordinate_norm_unit-balls_inclusions}
      and by the convention \( \CoordinateNorm{\TripleNormBall}{0} = \{0\} \)}
    &=
      \delta_{ \CoordinateNorm{\TripleNormBall}{k} } \circ \normalized
      \tag{because the unit ball \( \CoordinateNorm{\TripleNormBall}{k} \)
      is closed and convex}
    \\        
    &=
      \delta_{ \normalized^{-1}\np{\CoordinateNorm{\TripleNormBall}{k} } }
      \intertext{ where, by~\eqref{eq:normalization_mapping}, \( \normalized^{-1}\np{\CoordinateNorm{\TripleNormBall}{k} } =
      \{0\} \cup 
      \nset{ \primal \in \RR^d\backslash\{0\} }{\CoordinateNorm{\TripleNorm{\frac{\primal}{\TripleNorm{\primal}}}}{k}
      \leq 1 } \), so that we go on with}
    &=
      \delta_{ \nset{ \primal \in \RR^d }{ \CoordinateNorm{\TripleNorm{\primal}}{k}
      \leq \TripleNorm{\primal} } }
    \\        
    &=
      \delta_{ \nset{ \primal \in \RR^d }{ \CoordinateNorm{\TripleNorm{\primal}}{k}
      =  \TripleNorm{\primal} } }
      \tag{using the equality and inequalities between norms
      in~\eqref{eq:coordinate_norm_inequalities}}
  \end{align*}
  Therefore, we have
  \begin{align*}
    \forall k \in
    &\ic{0,d} \eqsepv 
      \SFMbi{ \delta_{ \LevelSet{\lzero}{k} } }{\CouplingCapra}
      = \delta_{ \LevelSet{\lzero}{k} }
    \\
    & \Leftrightarrow
      \forall k \in \ic{0,d} \eqsepv 
      \bgp{ \primal \in \LevelSet{\lzero}{k} \iff
      \CoordinateNorm{\TripleNorm{\primal}}{k}
      = \TripleNorm{\primal} \eqsepv \forall \primal \in \RR^d }
    \\
    & \Leftrightarrow
      \textrm{\eqref{eq:strictly_decreasingly_graded_b} holds true for
      the sequence 
      \(
      \sequence{\CoordinateNorm{\TripleNorm{\cdot}}{\LocalIndex}}{\LocalIndex\in\ic{1,d}}
      \)}
      \tag{because \( \primal \in \LevelSet{\lzero}{k} \iff \lzero\np{\primal} \leq k
      \) by definition of the level sets in~\eqref{eq:pseudonormlzero_level_set}}
    \\
    & \Leftrightarrow
      \textrm{  \( \sequence{\CoordinateNorm{\TripleNorm{\cdot}}{\LocalIndex}}{\LocalIndex\in\ic{1,d}} \)
      is strictly decreasingly graded \wrt\ the \lzeropseudonorm}
  \end{align*}
  because this sequence is nonincreasing
  by~\eqref{eq:dual_coordinate_norm_inequalities}
  (see Definition~\ref{de:decreasingly_graded}).
  \medskip

  This ends the proof.
\end{proof}

Notice that, by Item~\ref{it:coordinate_norm_strictly_graded} 
in Proposition~\ref{pr:coordinate_norm_graded},
it suffices that the normed space 
\( \bp{\RR^d,\TripleNorm{\cdot}} \) be strictly convex
to obtain that the characteristic functions \( \delta_{ \LevelSet{\lzero}{k} } \)
are $\CouplingCapra$-convex, for all $k\in\ic{0,d} $.
This is the case when the source norm is the $\ell_p$-norm~$\norm{\cdot}_{p}$
for $p\in ]1,\infty[$. 
\medskip

Determinining sufficient conditions under which the \lzeropseudonorm\
is $\CouplingCapra$-convex requires additional notions. 
This question is treated in the companion
paper~\cite{Chancelier-DeLara:2020_Variational}.

\subsection{\Capra-subdifferentials related to the \lzeropseudonorm}
\label{Capra-subdifferentials_related_to_the_lzeropseudonorm}

With the Fenchel conjugacy, we calculate that 
\( \partial\delta_{ \LevelSet{\lzero}{k} }\np{\primal} =\{0\} \)
for \( \primal \in \LevelSet{\lzero}{k} \) and $k\in\ic{1,d} $
(when \( \primal \not\in \LevelSet{\lzero}{k} \), 
\( \partial\delta_{ \LevelSet{\lzero}{k} }\np{\primal} =\emptyset \)).
We also calculate that \( \partial\lzero\np{0} =\{0\} \) 
and \( \partial\lzero\np{\primal} =\emptyset \),
for all \( \primal \in \RR^d\backslash\{0\} \)
(indeed, this is a consequence of 
\( \LFMbi{ \lzero }\np{\primal}= 0 \neq \lzero\np{\primal}\)
when \( \primal \in \RR^d\backslash\{0\} \)).
Hence, the Fenchel conjugacy is not suitable
to handle the \lzeropseudonorm.

By contrast, we will now show that functions of 
the \lzeropseudonorm\ in~\eqref{eq:pseudo_norm_l0}
--- including the \lzeropseudonorm\ itself and 
the characteristic functions \( \delta_{ \LevelSet{\lzero}{k} } \) 
of its level sets~\eqref{eq:pseudonormlzero_level_set} ---
display \Capra-subdifferentials, as in~\eqref{eq:Capra-subdifferential_b}, that are related to 
the sequence of dual coordinate-$k$ norms in Definition~\ref{de:coordinate_norm}
as follows. For this purpose, 
we recall that the \emph{normal cone}~$\NORMAL_{\Convex}(\primal)$ 
to the (nonempty) closed convex subset~${\Convex} \subset \RR^d $
at~$\primal \in \Convex$ is the closed convex cone defined by 
\cite[p.136]{Hiriart-Urruty-Lemarechal-I:1993}
\begin{equation}
  \NORMAL_{\Convex}(\primal) =
  \bset{ \dual \in \RR^d }%
  {
    \proscal{\primal'-\primal}{\dual} \leq 0 \eqsepv 
    \forall \primal' \in \Convex
  }
  \eqfinp
  \label{eq:normal_cone}
\end{equation}

\begin{proposition}
  Let $\TripleNorm{\cdot}$ be a norm on~$\RR^d$,
  with associated sequence
  \( \sequence{\CoordinateNorm{\TripleNormDual{\cdot}}{\LocalIndex}}{\LocalIndex\in\ic{1,d}} \)
  of dual coordinate-$k$ norms, as 
  in Definition~\ref{de:coordinate_norm},
  and 
  associated \Capra-coupling $\CouplingCapra$ in~\eqref{eq:coupling_CAPRA}. 

  Let \( \varphi : \ic{0,d} \to \barRR \) be a function 
  and $\primal \in \RR^d$ be a vector.

  \begin{itemize}
  \item 
    The \Capra-subdifferential, as in~\eqref{eq:Capra-subdifferential_zero}, of the function
    \( \varphi \circ \lzero \) at~$\primal=0$ is given by 
    \begin{equation}
      \partial_{\CouplingCapra}\np{ \varphi \circ \lzero}\np{0}
      = \bigcap_{ \LocalIndex\in\ic{1,d} } \bc{ \varphi\np{\LocalIndex} \UppPlus \bp{-\varphi\np{0} } } 
      \CoordinateNormDual{\TripleNormBall}{\LocalIndex} 
      \eqfinv
    \end{equation}
    where, by convention 
    \( \lambda \CoordinateNormDual{\TripleNormBall}{\LocalIndex} =\emptyset \),
    for any \( \lambda \in [-\infty,0[ \), and 
    \( +\infty\CoordinateNormDual{\TripleNormBall}{\LocalIndex} =\RR^d \).
  \item 
    The \Capra-subdifferential, as in~\eqref{eq:Capra-subdifferential_neq_zero}, of the function
    \( \varphi \circ \lzero \) at~$\primal\not=0$ is given by
    the following cases
    \begin{itemize}
    \item 
      if \( l=\lzero\np{\primal} \geq 1 \) 
      and either \( \varphi\np{l}=-\infty \)
      or \( \varphi \equiv +\infty \), 
      then \( \partial_{\CouplingCapra}\np{ \varphi \circ \lzero}\np{\primal} =\RR^d\),
    \item 
      if \( l=\lzero\np{\primal} \geq 1 \) 
      and \( \varphi\np{l}=+\infty \) and there exists 
      \( \LocalIndex \in \ic{0,d} \) such that 
      \( \varphi\np{\LocalIndex} \not= +\infty \), then 
      \( \partial_{\CouplingCapra}\np{ \varphi \circ \lzero}\np{\primal} =\emptyset \),
    \item 
      if \( l=\lzero\np{\primal} \geq 1 \) and \( -\infty < \varphi\np{l} < +\infty \), then 
      \begin{equation}
        \dual \in \partial_{\CouplingCapra}\np{ \varphi \circ \lzero}\np{\primal} 
        \iff 
        \begin{cases}
          \dual \in 
          \NORMAL_{ \CoordinateNorm{\TripleNormBall}{l} }
          \np{\frac{ \primal }{ \CoordinateNorm{ \TripleNorm{\primal} }{l} } }
          \quad \mtext{and }
          \\
          l \in \argmax_{\LocalIndex\in\ic{0,d}} 
          \bc{
            \CoordinateNormDual{\TripleNorm{\dual}}{\LocalIndex}-\varphi\np{\LocalIndex}
          } \eqfinp
        \end{cases}
        \label{eq:pseudonormlzero_subdifferential}
      \end{equation}
      %
    \end{itemize}
  \end{itemize}
  \label{pr:pseudonormlzero_subdifferential_varphi}
\end{proposition}

\begin{proof}
  We have
  \begin{align*}
    \dual \in 
    \partial_{\CouplingCapra}\np{ \varphi \circ \lzero}\np{\primal} 
    &\iff \SFM{ \np{ \varphi \circ \lzero} }{\CouplingCapra}\np{\dual} 
      = \CouplingCapra\np{\primal, \dual} 
      \LowPlus \bp{ -\np{ \varphi \circ \lzero}\np{\primal} }
      \intertext{by definition~\eqref{eq:Capra-subdifferential_b} of the 
      \Capra-subdifferential }
    &\iff 
      \sup_{\LocalIndex\in\ic{0,d}} \Bc{ \CoordinateNormDual{\TripleNorm{\dual}}{\LocalIndex} -\varphi\np{\LocalIndex} }  
      = \CouplingCapra\np{\primal, \dual} 
      \LowPlus \bp{ -\np{ \varphi \circ \lzero}\np{\primal} }
      \tag{as 
      \( \SFM{ \np{ \varphi \circ \lzero} }{\CouplingCapra}\np{\dual} 
      = \sup_{\LocalIndex\in\ic{0,d}} \Bc{ \CoordinateNormDual{\TripleNorm{\dual}}{\LocalIndex} -\varphi\np{\LocalIndex} } \) 
      by~\eqref{eq:conjugate_l0norm_varphi}
      }
    \\
    &\iff 
      \Bp{ \primal=0 \mtext{ and } 
      \sup_{\LocalIndex\in\ic{0,d}} \Bc{ \CoordinateNormDual{\TripleNorm{\dual}}{\LocalIndex} -\varphi\np{\LocalIndex} } =-\varphi\np{0} }
    \\
    & \mtext{ or } 
      \Bp{ \primal \neq 0 \mtext{ and } 
      \sup_{\LocalIndex\in\ic{0,d}} \Bc{ \CoordinateNormDual{\TripleNorm{\dual}}{\LocalIndex} -\varphi\np{\LocalIndex} }
      = \frac{ \proscal{\primal}{\dual} }{ \TripleNorm{\primal} } 
      - \varphi\bp{\lzero\np{\primal}} 
      }
      \tag{by definition~\eqref{eq:coupling_CAPRA} of \( \CouplingCapra\np{\primal, \dual} \)}
  \end{align*}
  \begin{subequations}
    Therefore, on the one hand, we obtain that
    \begin{align*}
      \dual \in \partial_{\CouplingCapra}\np{ \varphi \circ \lzero}\np{0} 
      &\iff 
        \CoordinateNormDual{\TripleNorm{\dual}}{\LocalIndex}
        -\varphi\np{\LocalIndex} \leq -\varphi\np{0} \eqsepv
        \forall \LocalIndex\in\ic{1,d}
        \tag{as \( \CoordinateNormDual{\TripleNorm{\dual}}{0}=0 \) 
        by convention} 
      \\
      &\iff 
        \CoordinateNormDual{\TripleNorm{\dual}}{\LocalIndex}
        \leq \varphi\np{\LocalIndex} \UppPlus \bp{-\varphi\np{0} } \eqsepv
        \forall \LocalIndex\in\ic{1,d} 
        \intertext{by property of the Moreau upper addition \cite{Moreau:1970}} 
      &\iff 
        \dual \in \bigcap_{ \LocalIndex\in\ic{1,d} } 
        \bc{ \varphi\np{\LocalIndex} \UppPlus \bp{-\varphi\np{0} } } 
        \CoordinateNormDual{\TripleNormBall}{\LocalIndex} 
        \eqfinv
    \end{align*}
    where, by convention 
    \( \lambda \CoordinateNormDual{\TripleNormBall}{\LocalIndex} =\emptyset \),
    for any \( \lambda \in [-\infty,0[ \), and 
    \( +\infty\CoordinateNormDual{\TripleNormBall}{\LocalIndex} =\RR^d \).

    On the other hand, when \( \primal \neq 0 \), we get
    \begin{equation}
      \dual \in \partial_{\CouplingCapra}\np{ \varphi \circ \lzero}\np{\primal} 
      \iff 
       \sup_{\LocalIndex\in\ic{0,d}} \bc{ \CoordinateNormDual{\TripleNorm{\dual}}{\LocalIndex}-\varphi\np{\LocalIndex} }=
      \frac{ \proscal{\primal}{\dual} }{ \TripleNorm{\primal} } 
      - \varphi\bp{\lzero\np{\primal}} 
      \eqfinp
      \label{eq:pseudonormlzero_subdifferential_proof_neq_0}
    \end{equation}
  \end{subequations}
  We now establish necessary and sufficient conditions for \( \dual \) to belong to
  \( \partial_{\CouplingCapra}\np{ \varphi \circ \lzero}\np{\primal} \) when \( \primal \neq 0 \).
  We consider \( \primal \in \RR^d\backslash\{0\} \), and we denote
  \( L = \Support{\primal} \) 
  and \( l=\cardinal{L}=\lzero\np{\primal} \).
  We have
  \begin{align*}
    \dual
    &\in 
      \partial_{\CouplingCapra}\np{ \varphi \circ \lzero}\np{\primal} 
     \\
    &\iff
      \sup_{\LocalIndex\in\ic{0,d}} \bc{ \CoordinateNormDual{\TripleNorm{\dual}}{\LocalIndex}-\varphi\np{\LocalIndex} }=
      \frac{ \proscal{\primal}{\dual} }{ \TripleNorm{\primal} } -\varphi\np{l}
      \tag{by~\eqref{eq:pseudonormlzero_subdifferential_proof_neq_0} with \( \lzero\np{\primal}=l \)}
    \\
    &\iff 
      \CoordinateNormDual{\TripleNorm{\dual}}{l} -\varphi\np{l} \leq
      \sup_{\LocalIndex\in\ic{0,d}} \bc{ \CoordinateNormDual{\TripleNorm{\dual}}{\LocalIndex}-\varphi\np{\LocalIndex} }=
      \frac{ \proscal{\primal}{\dual} }{ \TripleNorm{\primal} } -\varphi\np{l} 
    \\
    &\iff 
      \TripleNorm{\dual_L}_{L,\star} -\varphi\np{l} \leq 
      \CoordinateNormDual{\TripleNorm{\dual}}{l} -\varphi\np{l} \leq
      \sup_{\LocalIndex\in\ic{0,d}} \bc{ \CoordinateNormDual{\TripleNorm{\dual}}{\LocalIndex}-\varphi\np{\LocalIndex} }=
      \frac{ \proscal{\primal}{\dual} }{ \TripleNorm{\primal} } -\varphi\np{l} 
      \intertext{ as \( \TripleNorm{\dual_L}_{L,\star} \leq
      \CoordinateNormDual{\TripleNorm{\dual}}{l} \) 
      by expression~\eqref{eq:dual_coordinate_norm} of the dual coordinate-$k$ norm
      \( \CoordinateNormDual{\TripleNorm{\dual}}{l} \),
      and because \( l=\cardinal{L} \) }
    &\iff 
      \TripleNorm{\dual_L}_{L,\star} -\varphi\np{l} \leq 
      \CoordinateNormDual{\TripleNorm{\dual}}{l} -\varphi\np{l} \leq
      \sup_{\LocalIndex\in\ic{0,d}} \bc{ \CoordinateNormDual{\TripleNorm{\dual}}{\LocalIndex}-\varphi\np{\LocalIndex} }=
      \frac{ \proscal{\primal}{\dual} }{ \TripleNorm{\primal} } -\varphi\np{l} \leq
      \TripleNorm{\dual_L}_{L,\star} -\varphi\np{l} 
      \tag{as we have \( \frac{ \proscal{\primal}{\dual} }{ \TripleNorm{\primal} } 
      = \frac{ \proscal{\primal_L}{\dual_L} }{ \TripleNorm{\primal_L} } 
      \leq \TripleNorm{\dual_L}_{L,\star} \) since \( \primal=\primal_L\) and 
      by~\eqref{eq:norm_dual_norm_inequality} } 
   \\
    &\iff 
      \TripleNorm{\dual_L}_{L,\star} -\varphi\np{l} = 
      \CoordinateNormDual{\TripleNorm{\dual}}{l} -\varphi\np{l} =
      \sup_{\LocalIndex\in\ic{0,d}} \bc{ \CoordinateNormDual{\TripleNorm{\dual}}{\LocalIndex}-\varphi\np{\LocalIndex} }=
      \frac{ \proscal{\primal}{\dual} }{ \TripleNorm{\primal} } -\varphi\np{l} 
      \intertext{as all terms in the inequalities are necessarily equal }
    &\iff
      \begin{cases}
        \text{either } \varphi\np{l}=-\infty 
        \\
        \text{or } \bp{ \varphi\np{l}=+\infty \text{ and }
          \varphi\np{\LocalIndex}=+\infty \eqsepv \forall \LocalIndex\in\ic{0,d} }
        \\[2mm]
        \text{or }
        \Big( -\infty < \varphi\np{l} < +\infty \mtext{ and } 
        \\ \qquad 
        \TripleNorm{\dual_L}_{L,\star} = 
        \CoordinateNormDual{\TripleNorm{\dual}}{l} =
        \frac{ \proscal{\primal}{\dual} }{ \TripleNorm{\primal} }
        \mtext{ and } \CoordinateNormDual{\TripleNorm{\dual}}{l}-\varphi\np{l} = 
        \sup_{\LocalIndex\in\ic{0,d}} \bc{ \CoordinateNormDual{\TripleNorm{\dual}}{\LocalIndex}-\varphi\np{\LocalIndex} }
        \Big) \eqfinp
      \end{cases}
  \end{align*}
  Let us make a brief insert and notice that 
  \begin{align*}
    \primal=\primal_L &\eqsepv \lzero\np{\primal}=l=\cardinal{L}>1 \eqsepv
                        \proscal{\primal}{\dual} =
                        \TripleNorm{\primal} \times
                        \CoordinateNormDual{\TripleNorm{\dual}}{l} 
    \\
    \implies & \quad
                  \lzero\np{\primal}=l=\cardinal{L}>1 \eqsepv
                  \proscal{\primal_L}{\dual_L} =
                  \TripleNorm{\primal_L} \times
                  \CoordinateNormDual{\TripleNorm{\dual}}{l} 
    \\
    \implies & \quad
                  \lzero\np{\primal}=l=\cardinal{L}>1 \eqsepv
                  \TripleNorm{\primal_L} \times
                  \CoordinateNormDual{\TripleNorm{\dual}}{l} 
                  \leq
                  \TripleNorm{\primal_L} \times
                  \TripleNorm{\dual_L}_{L,\star} 
                  \tag{by~\eqref{eq:norm_dual_norm_inequality}}
    \\
    \implies & \quad
                  l=\cardinal{L} \eqsepv
                  \CoordinateNormDual{\TripleNorm{\dual}}{l} 
                  \leq
                  \TripleNorm{\dual_L}_{L,\star} 
    \\
    \implies & \quad
                  \CoordinateNormDual{\TripleNorm{\dual}}{l} 
                  = \TripleNorm{\dual_L}_{L,\star}
  \end{align*}              
  as \( \TripleNorm{\dual_L}_{L,\star} \leq
  \CoordinateNormDual{\TripleNorm{\dual}}{l} \) 
  by expression~\eqref{eq:dual_coordinate_norm} of the dual coordinate-$k$ norm
  \( \CoordinateNormDual{\TripleNorm{\dual}}{l} \),
  and because \( l=\cardinal{L} \).

  Now, let us go back to the equivalences regarding
  \(  \dual \in \partial_{\CouplingCapra}
  \np{ \varphi \circ \lzero}\np{\primal} \).
  Focusing on the case where \( -\infty < \varphi\np{l} < +\infty \), 
  we have 
  \begin{align*}
    \dual \in \partial_{\CouplingCapra}
    &\np{ \varphi \circ \lzero}\np{\primal}
      \Leftrightarrow 
      \TripleNorm{\dual_L}_{L,\star} = 
      \CoordinateNormDual{\TripleNorm{\dual}}{l} =
      \frac{ \proscal{\primal}{\dual} }{ \TripleNorm{\primal} }
      \mtext{ and } 
      l \in \argmax_{\LocalIndex\in\ic{0,d}} \bc{ \CoordinateNormDual{\TripleNorm{\dual}}{\LocalIndex}-\varphi\np{\LocalIndex} }
    \\ 
    &\Leftrightarrow 
      \TripleNorm{\dual_L}_{L,\star} = 
      \CoordinateNormDual{\TripleNorm{\dual}}{l} 
      \mtext{ and } 
      \proscal{\primal}{\dual} =
      \TripleNorm{\primal} \times
      \CoordinateNormDual{\TripleNorm{\dual}}{l} 
      \mtext{ and } 
      l \in \argmax_{\LocalIndex\in\ic{0,d}} \bc{ \CoordinateNormDual{\TripleNorm{\dual}}{\LocalIndex}-\varphi\np{\LocalIndex} }
    \\ 
    &\Leftrightarrow 
      \proscal{\primal}{\dual} =
      \TripleNorm{\primal} \times
      \CoordinateNormDual{\TripleNorm{\dual}}{l} 
      \mtext{ and } 
      l \in \argmax_{\LocalIndex\in\ic{0,d}} \bc{ \CoordinateNormDual{\TripleNorm{\dual}}{\LocalIndex}-\varphi\np{\LocalIndex} }
      \intertext{as just established in the insert}
    &\Leftrightarrow 
      \proscal{\primal}{\dual} =
      \CoordinateNorm{\TripleNorm{\primal}}{l} \times
      \CoordinateNormDual{\TripleNorm{\dual}}{l} 
      \mtext{ and } 
      l \in \argmax_{\LocalIndex\in\ic{0,d}} \bc{ \CoordinateNormDual{\TripleNorm{\dual}}{\LocalIndex}-\varphi\np{\LocalIndex} }
      \tag{as \( \lzero\np{\primal}=l \implies 
      \TripleNorm{\primal}=
      \CoordinateNorm{\TripleNorm{\primal}}{l} \) by~\eqref{eq:coordinate_norm_graded_b}}
    \\
    &\Leftrightarrow 
      \dual \in \NORMAL_{ \CoordinateNorm{\TripleNormBall}{l} }
      \np{ \frac{ \primal }{ \CoordinateNorm{\TripleNorm{\primal}}{l} } }
      \mtext{ and } 
      l \in \argmax_{\LocalIndex\in\ic{0,d}} \bc{ \CoordinateNormDual{\TripleNorm{\dual}}{\LocalIndex}-\varphi\np{\LocalIndex} }
      %
  \end{align*}
  by the equivalence
  \( \proscal{\primal}{\dual} =
      \CoordinateNorm{\TripleNorm{\primal}}{l} \times
      \CoordinateNormDual{\TripleNorm{\dual}}{l} \iff
      \dual \in \NORMAL_{ \CoordinateNorm{\TripleNormBall}{l} }
      \np{ \frac{ \primal }{ \CoordinateNorm{\TripleNorm{\primal}}{l} } } \).
      \medskip
      
  This ends the proof. 
\end{proof}

With \( \varphi \) the identity function on~\( \ic{0,d} \),
we find the \Capra-subdifferential of the \lzeropseudonorm.
With the functions
\( \varphi=\delta_{\ic{0,k}} \) (for any $k \in \ic{0,d} $),
we find the \Capra-subdifferentials of the
characteristic functions \( \delta_{ \LevelSet{\lzero}{k} } \)
of its level sets~\eqref{eq:pseudonormlzero_level_set}.
The corresponding expressions are given in Table~\ref{tab:results_conjugacy}.

\section{Norm ratio lower bounds for the $l_0$ pseudonorm}
\label{Norm_ratio_lower_bounds_for_the_l0_pseudonorm}

As an application, we provide a new family of lower bounds for 
the \lzeropseudonorm, as a fraction between two norms,
the denominator being any norm.

\begin{proposition}
  Let $\TripleNorm{\cdot}$ be a norm on~$\RR^d$,
  with associated sequence of dual coordinate-$k$ norms, as 
  in Definition~\ref{de:coordinate_norm}. 
  \begin{subequations}
    For any function \( \varphi : \ic{0,d} \to [0,+\infty[ \), 
    such that \( \varphi\np{\LocalIndex} > \varphi\np{0}=0 \) for all \( \LocalIndex\in\ic{1,d} \), 
    there exists a norm \( \CoordinateNorm{\TripleNorm{\cdot}}{\varphi} \)
    characterized 
    \begin{itemize}
    \item 
      either by its dual norm
      \( \CoordinateNormDual{\TripleNorm{\cdot}}{\varphi} \), which has unit ball
      \( \bigcap_{ \LocalIndex\in\ic{1,d} } \varphi\np{\LocalIndex} \CoordinateNormDual{\TripleNormBall}{\LocalIndex} \),
      that is,
      \begin{align}
          \CoordinateNormDual{\TripleNormBall}{\varphi}
        &=
        \bigcap_{ \LocalIndex\in\ic{1,d} } \varphi\np{\LocalIndex} \CoordinateNormDual{\TripleNormBall}{\LocalIndex}
        \mtext{ and }
        \CoordinateNorm{\TripleNorm{\cdot}}{\varphi} 
        =
        \sigma_{ \CoordinateNormDual{\TripleNormBall}{\varphi} }
        \eqfinv
        \label{eq:coordinate_norm_varphi}
   \\
\text{or, equivalently,} \qquad   
        \CoordinateNormDual{\TripleNorm{\dual}}{\varphi}
        &=
        \sup_{ \LocalIndex\in\ic{1,d} }
        \frac{\CoordinateNormDual{\TripleNorm{\dual}}{\LocalIndex}}%
        { \varphi\np{\LocalIndex} }
        \eqsepv \forall \dual \in \RR^d 
        \eqfinv
      \end{align}

    \item 
      or by the inf-convolution
      \begin{align}
        \CoordinateNorm{\TripleNorm{\cdot}}{\varphi} 
        &=  \bigbox_{ \LocalIndex\in\ic{1,d} }
        \Bp{ \varphi\np{\LocalIndex} \CoordinateNorm{\TripleNorm{\cdot}}{\LocalIndex} }
        \eqfinv
  \\
\text{that is,} \qquad
        \CoordinateNorm{\TripleNorm{\primal}}{\varphi} 
        &=
        \inf_{ \substack{%
            z^{(1)} \in \RR^d, \ldots, z^{(d)} \in \RR^d
            \\
            \sum_{ \LocalIndex=1 }^{ d } z^{(\LocalIndex)} = \primal  } }
        \sum_{ \LocalIndex=1 }^{ d }  \varphi\np{\LocalIndex} 
        \CoordinateNorm{\TripleNorm{z^{(\LocalIndex)}}}{\LocalIndex} 
        \eqsepv \forall \primal \in \RR^d 
        \eqfinp
        \label{eq:coordinate_norm_varphi_inf-convolution}
        \end{align}
      %
    \end{itemize}
  \end{subequations} 
Then, we have the inequalities
  \begin{equation}
    \frac{ \CoordinateNorm{\TripleNorm{\primal}}{\varphi} }{ \TripleNorm{\primal} }
    \leq 
    \frac{ 1 }{ \TripleNorm{\primal} } 
    \min_{ \substack{%
        z^{(1)} \in \RR^d, \ldots, z^{(d)} \in \RR^d 
        \\
        \sum_{ \LocalIndex=1 }^{ d } \CoordinateNorm{\TripleNorm{z^{(\LocalIndex)}}}{\LocalIndex} \leq \TripleNorm{\primal}
        \\
        \sum_{ \LocalIndex=1 }^{ d } z^{(\LocalIndex)} = \primal  } }
    \sum_{ \LocalIndex=1 }^{ d } \varphi\np{\LocalIndex} 
    \CoordinateNorm{\TripleNorm{z^{(\LocalIndex)}}}{\LocalIndex} 
    \leq 
    \varphi\bp{ \lzero\np{\primal} }
    \eqsepv \forall \primal \in \RR^d\backslash\{0\}
    \eqfinp
    \label{eq:coordinate_norm_varphi_inequality_for_lzero}
  \end{equation}
  \label{pr:coordinate_norm_varphi}
\end{proposition}

\begin{proof}
  \quad
  
  \noindent $\bullet$ 
  It is easily seen that 
  \( \sigma_{ \CoordinateNormDual{\TripleNormBall}{\varphi} } \) in~\eqref{eq:coordinate_norm_varphi}
  defines a norm, and that, for all \( \dual \in \RR^d \),
  \[
    \CoordinateNormDual{\TripleNorm{\dual}}{\varphi}
    = \inf \bset{ \lambda \geq 0 }{ \dual \in \lambda 
      \bigcap_{ \LocalIndex=1}^{d} \varphi\np{\LocalIndex}
      \CoordinateNormDual{\TripleNormBall}{\LocalIndex} }
    = \inf \bset{ \lambda  \geq 0 }{
      \frac{\CoordinateNormDual{\TripleNorm{\dual}}{\LocalIndex}}%
      { \varphi\np{\LocalIndex} } \leq \lambda }
    =
    \sup_{ \LocalIndex\in\ic{1,d} }
    \frac{\CoordinateNormDual{\TripleNorm{\dual}}{\LocalIndex}}%
    { \varphi\np{\LocalIndex} }
    \eqfinp
  \]

  \medskip

  \noindent $\bullet$ 
  We have 
    \begin{align*}
  \CoordinateNorm{\TripleNorm{\cdot}}{\varphi} 
    &=
      \sigma_{ \CoordinateNormDual{\TripleNormBall}{\varphi} } 
      \tag{by~\eqref{eq:coordinate_norm_varphi}}
    \\
    &=
      \LFM{ \delta_{ \CoordinateNormDual{\TripleNormBall}{\varphi} } }
      \tag{because \( \CoordinateNormDual{\TripleNormBall}{\varphi} \) is closed and convex}
    \\
    &=
      \LFM{ \bp{ 
      \sum_{ \LocalIndex\in\ic{1,d} } \delta_{ \varphi\np{\LocalIndex} \CoordinateNormDual{\TripleNormBall}{\LocalIndex} } } }
      \intertext{by~\eqref{eq:coordinate_norm_varphi} and by expressing the characteristic function of an
      intersection of sets as a sum}
    &=
      \bigbox_{ \LocalIndex\in\ic{1,d} } \LFM{ \delta_{ \varphi\np{\LocalIndex}  \CoordinateNormDual{\TripleNormBall}{\LocalIndex} } }
      \intertext{using~\cite[Proposition~15.3 and (v) in Proposition-15.5]{Bauschke-Combettes:2017}
      because the intersection \( \CoordinateNormDual{\TripleNormBall}{\varphi}
      =
      \bigcap_{ \LocalIndex=1}^{d} \varphi\np{\LocalIndex} 
      \CoordinateNormDual{\TripleNormBall}{\LocalIndex} \)
      of all the domains of the functions
      \( \delta_{ \varphi\np{\LocalIndex}  \CoordinateNormDual{\TripleNormBall}{\LocalIndex} } \)
      contain a neighborhood of~$0$ since \( \varphi\np{\LocalIndex} >0 \)
      for all \( \LocalIndex\in\ic{1,d} \)}
    &=
      \bigbox_{ \LocalIndex\in\ic{1,d} } \sigma_{ \varphi\np{\LocalIndex}
      \CoordinateNormDual{\TripleNormBall}{\LocalIndex} }
      \tag{as \( \LFM{ \delta_{ \varphi\np{\LocalIndex}
      \CoordinateNormDual{\TripleNormBall}{\LocalIndex} } }
      = \sigma_{ \varphi\np{\LocalIndex}
      \CoordinateNormDual{\TripleNormBall}{\LocalIndex} } \),
      for all \( \LocalIndex\in\ic{1,d} \) }
    \\
    &=
      \bigbox_{ \LocalIndex\in\ic{1,d} } \varphi\np{\LocalIndex} \CoordinateNorm{\TripleNorm{\cdot}}{\LocalIndex} 
      \tag{by~\eqref{eq:norm_dual_norm}}
  \end{align*}
  \medskip

  \noindent $\bullet$ 
  We consider the coupling $\CouplingCapra$ in~\eqref{eq:coupling_CAPRA}. 

  By~\eqref{eq:biconjugate_l0norm_varphi} --- because
  the function \( \varphi : \ic{0,d} \to [0,+\infty[ \)
  satisfies the assumption in Item~\ref{it:biconjugate_l0norm_varphi}
  of Proposition~\ref{pr:pseudonormlzero_biconjugate_varphi} ---
  and by the inequality
  \( \SFMbi{\np{ \varphi \circ \lzero } }{\CouplingCapra}
  \leq \varphi \circ \lzero \) obtained from~\eqref{eq:galois-cor},
  we get that
  \begin{equation}
    \frac{ 1 }{ \TripleNorm{\primal} } 
    \min_{ \substack{%
        z^{(1)} \in \RR^d, \ldots, z^{(d)} \in \RR^d 
        \\
        \sum_{ \LocalIndex=1 }^{ d } \CoordinateNorm{\TripleNorm{z^{(\LocalIndex)}}}{\LocalIndex} \leq \TripleNorm{\primal}
        \\
        \sum_{ \LocalIndex=1 }^{ d } z^{(\LocalIndex)} = \primal  } }
    \sum_{ \LocalIndex=1 }^{ d } \LocalIndex 
    \CoordinateNorm{\TripleNorm{z^{(\LocalIndex)}}}{\LocalIndex} 
    \leq 
    \varphi\bp{ \lzero\np{\primal} }
    \eqsepv \forall \primal \in \RR^d\backslash\{0\} 
    \eqfinp
    \label{eq:weaker_Variational_formulation_for_the_pseudo_norm}
  \end{equation}
  Thus, we have obtained the right hand side inequality
  in~\eqref{eq:coordinate_norm_varphi_inequality_for_lzero}.

  By relaxing one constraint in~\eqref{eq:weaker_Variational_formulation_for_the_pseudo_norm},
  we immediately get that 
  \[
    \inf_{ \substack{%
        z^{(1)} \in \RR^d, \ldots, z^{(d)} \in \RR^d
        \\
        \sum_{ \LocalIndex=1 }^{ d } z^{(\LocalIndex)} = \primal  } }
    \sum_{ \LocalIndex=1 }^{ d }  \varphi\np{\LocalIndex} 
    \CoordinateNorm{\TripleNorm{z^{(\LocalIndex)}}}{\LocalIndex} 
    \leq
    \min_{ \substack{%
        z^{(1)} \in \RR^d, \ldots, z^{(d)} \in \RR^d 
        \\
        \sum_{ \LocalIndex=1 }^{ d } \CoordinateNorm{\TripleNorm{z^{(\LocalIndex)}}}{\LocalIndex} \leq \TripleNorm{\primal}
        \\
        \sum_{ \LocalIndex=1 }^{ d } z^{(\LocalIndex)} = \primal  } }
    \sum_{ \LocalIndex=1 }^{ d } \varphi\np{\LocalIndex} 
    \CoordinateNorm{\TripleNorm{z^{(\LocalIndex)}}}{\LocalIndex} 
    \leq 
    \varphi\bp{ \lzero\np{\primal} }
    \eqsepv \forall \primal \in \RR^d
    \eqfinp 
  \]
  Thus, we have obtained the left hand side inequality
  in~\eqref{eq:coordinate_norm_varphi_inequality_for_lzero},
  thanks to~\eqref{eq:coordinate_norm_varphi_inf-convolution}.
\end{proof}

For any function \( \varphi : \ic{0,d} \to [0,+\infty[ \), 
such that \( \varphi\np{\LocalIndex} > \varphi\np{0}=0 \) for all \( \LocalIndex\in\ic{1,d} \), 
using Table~\ref{tab:Examples} when the source norm
$\TripleNorm{\cdot}$ is the $\ell_p$-norm \( \Norm{\cdot}_{p} \),
for \(p \in [1,\infty] \) and \( 1/p + 1/q =1 \),
we denote 
\( \CoordinateNorm{\TripleNorm{\cdot}}{\varphi} \)
by \( \Norm{\cdot}_{p,\varphi}^{\mathrm{sn}} \).
The calculations show that 
\( \Norm{\cdot}_{1,\varphi}^{\mathrm{sn}} = \Norm{\cdot}_{1}\),
and that, when \(p \in ]1,\infty] \), we also have
\( \Norm{\cdot}_{p,\varphi}^{\mathrm{sn}} = \Norm{\cdot}_{1}\),
whatever \(p \in [1,\infty] \),
if we suppose that \( \bp{\varphi\np{\LocalIndex}}^q \geq \LocalIndex \),
for all \( \LocalIndex\in\ic{1,d} \).
As a consequence, when $p=1$, the
inequality~\eqref{eq:coordinate_norm_varphi_inequality_for_lzero}
is trivial.
When \(p \in ]1,\infty] \), if we take the function
\( \varphi\np{\LocalIndex} = \LocalIndex^{1/q}\)
for all \( \LocalIndex\in\ic{1,d} \),
the
inequality~\eqref{eq:coordinate_norm_varphi_inequality_for_lzero}
yields that 
\( \frac{ \Norm{\primal}_{1} }{ \Norm{\primal}_{p} } \leq
\bp{ \lzero\np{\primal} }^{1/q} \),
which is easily obtained directly from the H\"{o}lder inequality.

\begin{table}[htbp]
  \centering
  \begin{tabular}{||l|l||}
    \hline\hline Fenchel conjugacy & \Capra\ conjugacy \\
    \hline\hline
    \( \minusLFM{ \delta_{ \LevelSet{\lzero}{k} } }=\LFM{ \delta_{ \LevelSet{\lzero}{k} } }= \delta_{\{0\}} \) 
                                   &
                                     \(       \SFM{ \delta_{
                                     \LevelSet{\lzero}{k} } }{-\couplingCAPRA}=
\SFM{ \delta_{ \LevelSet{\lzero}{k} } }{\couplingCAPRA}
                                     =        \CoordinateNormDual{\TripleNorm{\cdot}}{k} \)
    \\
    \hline
    \( \LFMbi{ \delta_{ \LevelSet{\lzero}{k} } }= 0 \) 
                                   &
                                     \( \SFMbi{ \delta_{ \LevelSet{\lzero}{k} } }{\CouplingCapra} 
                                     =
                                     \delta_{ \nset{ \primal \in \RR^d }{ \CoordinateNorm{\TripleNorm{\primal}}{k}
                                     =  \TripleNorm{\primal} } } \)
    \\
    \hline
    \( \partial\delta_{ \LevelSet{\lzero}{k} }\np{\primal} =\emptyset \)
                                   &
                                     \( \partial_{\CouplingCapra}\delta_{ \LevelSet{\lzero}{k} }\np{\primal} =
                                     \begin{cases}
                                       \emptyset 
                                       & \text{if } \lzero\np{\primal} = k+1, \ldots, d 
                                       \eqfinv
                                       \\
                                       \NORMAL_{ \CoordinateNorm{\TripleNormBall}{k} }
                                       \np{\frac{ \primal }{ \CoordinateNorm{ \TripleNorm{\primal} }{k} } }
                                       & \text{if } \lzero\np{\primal} =1, \ldots, k 
                                       \eqfinv
                                       \\
                                       \{0\}
                                       & \text{if } \lzero\np{\primal} = 0
                                     \end{cases} \)
    \\
    \( \forall \primal \in \RR^d \)
                                   &  \( \forall \primal \in \RR^d \)
    \\
    \hline\hline 
    \( \LFM{ \lzero }= \delta_{\{0\}} \)
                                   & 
                                     \( \SFM{ \lzero }{\couplingCAPRA}
                                     = \sup_{\LocalIndex\in\ic{0,d}} \bc{ \CoordinateNormDual{\TripleNorm{\cdot}}{\LocalIndex} - \LocalIndex }
                                     \)
    \\
    \hline
    \(\LFMbi{ \lzero } = 0 \)
                                   & 
                                     \( \SFMbi{ \lzero }{\CouplingCapra}\np{\primal}
                                     =
                                     \frac{ 1 }{ \TripleNorm{\primal} } 
                                     \min_{ \substack{%
                                     z^{(1)} \in \RR^d, \ldots, z^{(d)} \in \RR^d 
    \\
    \sum_{ \LocalIndex=1 }^{ d } \CoordinateNorm{\TripleNorm{z^{(\LocalIndex)}}}{\LocalIndex} \leq \TripleNorm{\primal}
    \\
    \sum_{ \LocalIndex=1 }^{ d } z^{(\LocalIndex)} = \primal  } }
    \sum_{ \LocalIndex=1 }^{ d } \LocalIndex
    \CoordinateNorm{\TripleNorm{z^{(\LocalIndex)}}}{\LocalIndex} 
    \eqsepv \forall \primal \in \RR^d\backslash\{0\} \) 
    \\
                                   &  
                                     \( \SFMbi{ \lzero }{\CouplingCapra}\np{0}=0 \)
    \\
    \hline
    \( \partial\lzero\np{0} =\{0\} \) 
                                   &
                                     \( \partial_{\CouplingCapra}\lzero\np{0}
                                     = \bigcap_{ \LocalIndex\in\ic{1,d} } \LocalIndex
                                     \CoordinateNormDual{\TripleNormBall}{\LocalIndex}
                                     = \CoordinateNormDual{\TripleNormBall}{\textrm{Id}} \)   
    \\
                                   &
    \\
    \( \partial\lzero\np{\primal} =\emptyset \) 
                                   &
                                     \(  \dual \in \partial_{\CouplingCapra}\lzero\np{\primal} \iff
                                     \begin{cases}
                                       \dual \in 
                                       \NORMAL_{ \CoordinateNorm{\TripleNormBall}{l} }
                                       \np{\frac{ \primal }{ \CoordinateNorm{ \TripleNorm{\primal} }{l} } }
                                       \\
                                       \mtext{and }
                                       l \in \argmax_{\LocalIndex\in\ic{0,d}} \bc{ \CoordinateNormDual{\TripleNorm{\dual}}{\LocalIndex}-\LocalIndex }
                                     \end{cases}
    \)
    \\
    \( \forall \primal \in \RR^d\backslash\{0\} \)
                                   & \( \forall \primal \in \RR^d\backslash\{0\} \), where \( l=\lzero\np{\primal} \geq 1 \) 
    \\
    \hline\hline
  \end{tabular}
  \caption{Comparison of Fenchel and \Capra-conjugates, biconjugates
    and subdifferentials
    of the \lzeropseudonorm\ in~\eqref{eq:pseudo_norm_l0}, and
    of the characteristic functions \( \delta_{ \LevelSet{\lzero}{k} } \) 
    of its level sets~\eqref{eq:pseudonormlzero_level_set}, 
    for \( k\in\ic{0,d} \) 
    \label{tab:results_conjugacy}}
\end{table}

\section{Conclusion}

In this paper, we have presented a new family of conjugacies, 
which depend on a given general 
source norm, and we have shown that they are suitable for the \lzeropseudonorm. 
More precisely, given a (source) norm on~$\RR^d$, we have defined,
on the one hand, a sequence of so-called coordinate-$k$ norms
and, on the other hand, a coupling between~$\RR^d$ and itself, 
called Capra (constant along primal rays).
With this, we have provided formulas for the \Capra-conjugate and biconjugate,
and for the \Capra\ subdifferentials, 
of functions of the \lzeropseudonorm,
in terms of the coordinate-$k$ norms.
Table~\ref{tab:results_conjugacy} provides the results of 
Proposition~\ref{pr:pseudonormlzero_conjugate_varphi},
Proposition~\ref{pr:pseudonormlzero_biconjugate_varphi},
and
Proposition~\ref{pr:pseudonormlzero_subdifferential_varphi},
in the case of the \lzeropseudonorm\ and of the
characteristic functions \( \delta_{ \LevelSet{\lzero}{k} } \)
of its level sets~\eqref{eq:pseudonormlzero_level_set}.
It compares them with the Fenchel conjugates and biconjugates.
As an application, we have provided a new family of lower bounds for 
the \lzeropseudonorm, as a fraction between two norms,
the denominator being any norm.

In the companion paper~\cite{Chancelier-DeLara:2020_Variational}, 
we provide sufficient conditions under which the \lzeropseudonorm\
is $\CouplingCapra$-convex. We are currently investigating how 
the Capra conjugacies could provide algorithms for exact sparse optimization
\medskip

\textbf{Acknowledgements.}
We want to thank Guillaume Obozinski
for discussions on first versions of this work,
as well as the anonymous Referee and Associate Editor
whose comments helped us improve the manuscript.

\appendix

\section{Background on Fenchel-Moreau conjugacies}
\label{Appendix}

We review general concepts and notations on Fenchel-Moreau conjugacies, 
then focus on the special case of the Fenchel conjugacy.

\subsubsubsection{The general case}

Let $\PRIMAL$ (``primal''), $\DUAL$ (``dual'') be two sets 
and \( \coupling : \PRIMAL \times \DUAL \to \barRR \) 
be a so-called \emph{coupling} function.
With any coupling, we associate \emph{conjugacies} 
from \( \barRR^\PRIMAL \) to \( \barRR^\DUAL \) 
and from \( \barRR^\DUAL \) to \( \barRR^\PRIMAL \) 
as follows.

\begin{subequations}
  %
    The \emph{$\coupling$-Fenchel-Moreau conjugate} of a 
    function \( \fonctionprimal : \PRIMAL  \to \barRR \), 
    with respect to the coupling~$\coupling$, is
    the function \( \SFM{\fonctionprimal}{\coupling} : \DUAL  \to \barRR \) 
    defined by
    \begin{equation}
      \SFM{\fonctionprimal}{\coupling}\np{\dual} = 
      \sup_{\primal \in \PRIMAL} \Bp{ \coupling\np{\primal,\dual} 
        \LowPlus \bp{ -\fonctionprimal\np{\primal} } } 
      \eqsepv \forall \dual \in \DUAL
      \eqfinp
      \label{eq:Fenchel-Moreau_conjugate}
    \end{equation}
    With the coupling $\coupling$, we associate 
    the \emph{reverse coupling~$\coupling'$} defined by 
    \begin{equation}
      \coupling': \DUAL \times \PRIMAL \to \barRR 
      \eqsepv
      \coupling'\np{\dual,\primal}= \coupling\np{\primal,\dual} 
      \eqsepv
      \forall \np{\dual,\primal} \in \DUAL \times \PRIMAL
      \eqfinp
      \label{eq:reverse_coupling}
    \end{equation}
    The \emph{$\coupling'$-Fenchel-Moreau conjugate} of a 
    function \( \fonctiondual : \DUAL \to \barRR \), 
    with respect to the coupling~$\coupling'$, is
    the function \( \SFM{\fonctiondual}{\coupling'} : \PRIMAL \to \barRR \) 
    defined by
    \begin{equation}
      \SFM{\fonctiondual}{\coupling'}\np{\primal} = 
      \sup_{ \dual \in \DUAL } \Bp{ \coupling\np{\primal,\dual} 
        \LowPlus \bp{ -\fonctiondual\np{\dual} } } 
      \eqsepv \forall \primal \in \PRIMAL 
      \eqfinp
      \label{eq:Fenchel-Moreau_reverse_conjugate}
    \end{equation}
    The \emph{$\coupling$-Fenchel-Moreau biconjugate} of a 
    function \( \fonctionprimal : \PRIMAL  \to \barRR \), 
    with respect to the coupling~$\coupling$, is
    the function \( \SFMbi{\fonctionprimal}{\coupling} : \PRIMAL \to \barRR \) 
    defined by
    \begin{equation}
      \SFMbi{\fonctionprimal}{\coupling}\np{\primal} = 
      \bp{\SFM{\fonctionprimal}{\coupling}}^{\coupling'} \np{\primal} = 
      \sup_{ \dual \in \DUAL } \Bp{ \coupling\np{\primal,\dual} 
        \LowPlus \bp{ -\SFM{\fonctionprimal}{\coupling}\np{\dual} } } 
      \eqsepv \forall \primal \in \PRIMAL 
      \eqfinp
      \label{eq:Fenchel-Moreau_biconjugate}
    \end{equation}
    %

  The biconjugate of a 
  function \( \fonctionprimal : \PRIMAL  \to \barRR \) satisfies
  \begin{equation}
    \SFMbi{\fonctionprimal}{\coupling}\np{\primal}
    \leq \fonctionprimal\np{\primal}
    \eqsepv \forall \primal \in \PRIMAL 
    \eqfinp
    \label{eq:galois-cor}
  \end{equation}
  
\end{subequations}

\subsubsubsection{The Fenchel conjugacy}

When the sets $\PRIMAL$ and $\DUAL$ are two vector spaces that
are \emph{paired} with a bilinear form \( \proscal{}{} \),
in the sense of convex analysis \cite[p.~13]{Rockafellar:1974},
the corresponding conjugacy is the classical 
\emph{Fenchel conjugacy}.
For any functions \( \fonctionprimal : \PRIMAL  \to \barRR \)
and \( \fonctiondual : \DUAL \to \barRR \), we denote\footnote{%
  In convex analysis, one does not use the notation~\( \LFMr{} \),
  but simply the notation~\( \LFM{} \), as it is often the case that 
  \(\PRIMAL=\DUAL\) in the Euclidian and Hilbertian cases.}
\begin{subequations}
  \begin{align}
    \LFM{\fonctionprimal}\np{\dual} 
    &= 
      \sup_{\primal \in \PRIMAL} \Bp{ \proscal{\primal}{\dual} 
      \LowPlus \bp{ -\fonctionprimal\np{\primal} } } 
      \eqsepv \forall \dual \in \DUAL
      \eqfinv
      \label{eq:Fenchel_conjugate}
    \\
    \LFMr{\fonctiondual}\np{\primal} 
    &= 
      \sup_{ \dual \in \DUAL } \Bp{ \proscal{\primal}{\dual} 
      \LowPlus \bp{ -\fonctiondual\np{\dual} } } 
      \eqsepv \forall \primal \in \PRIMAL 
      \eqfinv
      \label{eq:Fenchel_conjugate_reverse}
    \\
    \LFMbi{\fonctionprimal}\np{\primal} 
    &= 
      \sup_{\dual \in \DUAL} \Bp{ \proscal{\primal}{\dual} 
      \LowPlus \bp{ -\LFM{\fonctionprimal}\np{\dual} } } 
      \eqsepv \forall \primal \in \PRIMAL
      \eqfinp
      \label{eq:Fenchel_biconjugate}
  \end{align}
\end{subequations}

For any function \( \fonctionuncertain : \UNCERTAIN \to \barRR \),
its \emph{epigraph} is \( \epigraph\fonctionuncertain= 
\defset{ \np{\uncertain,t}\in\UNCERTAIN\times\RR}%
{\fonctionuncertain\np{\uncertain} \leq t} \),
its \emph{effective domain} is 
\( \dom\fonctionuncertain= 
\defset{\uncertain\in\UNCERTAIN}{ \fonctionuncertain\np{\uncertain} <+\infty}
\).
A function \( \fonctionuncertain : \UNCERTAIN \to \barRR \)
is said to be \emph{proper} if it never takes the value~$-\infty$
and that \( \dom\fonctionuncertain \not = \emptyset \).
When \( \UNCERTAIN \) is equipped with a topology,
the function \( \fonctionuncertain : \UNCERTAIN \to \barRR \)
is said to be \emph{lower semi continuous (lsc)}
if its epigraph is closed,
and is said to be \emph{closed}
if \( \fonctionuncertain \) is either \emph{lower semi continuous (lsc)}
and nowhere having the value $-\infty$,
or is the constant function~$-\infty$
\cite[p.~15]{Rockafellar:1974}.

It is proved that the Fenchel conjugacy induces a one-to-one correspondence
between the closed convex functions on~$\PRIMAL$
and the closed convex functions on~$\DUAL$
\cite[Theorem~5]{Rockafellar:1974}.
Here, a function  is said to be \emph{convex} if its
epigraph is convex.
The set of closed convex functions is
the set of proper convex functions united with the
two constant functions~$-\infty$ and ~$+\infty$.

\newcommand{\noopsort}[1]{} \ifx\undefined\allcaps\def\allcaps#1{#1}\fi

\end{document}